\documentclass{article}
\usepackage{amssymb, latexsym,pdfsync,amsmath,amsthm,hyperref,graphicx}
\usepackage{fullpage}
\usepackage{pgf,tikz}
\usepackage{mathrsfs}
\usepackage{enumerate}
\usepackage{authblk}
\usetikzlibrary{arrows}
\newtheorem{theorem}{Theorem}[section]
\newtheorem{corollary}[theorem]{Corollary}
\newtheorem{lemma}[theorem]{Lemma}

\theoremstyle{definition}
\newtheorem{definition}[theorem]{Definition}
\newtheorem{remark}[theorem]{Remark}

\newcommand{\NN}{\mathbb{N}}

\newcommand{\ZZ}{\mathbb{Z}}

\newcommand{\FF}{\mathbb{F}}

\newcommand{\Fq}{\mathbb{F}_q}
\newcommand{\Fqn}{\mathbb{F}_{q^n}}

\newcommand{\Fqtm}{\mathbb{F}_{q^{2m}}}

\newcommand{\Fqt}{\mathbb{F}_{q^2}}

\newcommand{\D}{\mathcal D}

\newcommand{\cI}{\mathcal I}
\newcommand{\cP}{\mathcal P}
\newcommand{\cL}{\mathcal L}

\def\S{\mathcal{S}}

\def\F{\mathbb{F}}

\def\Fq{{\mathbb{F}}_q}

\def\Aut{\mathrm{Aut}}

\def\PG{\mathrm{PG}}

\def\PGL{\mathrm{PGL}}
\def\GL{\mathrm{GL}}
\def\PGammaL{\mathrm{P\Gamma L}}
\def\AGammaL{\mathrm{A\Gamma L}}
\def\GammaL{\mathrm{\Gamma L}}

\newcommand{\npmatrix}[1]{\left( \begin{matrix} #1 \end{matrix} \right)}

\begin{document}
\title{Cyclic $2$-Spreads in $V(6,q)$ and Flag-Transitive Affine Linear Spaces}
\date{}
\author[$\ddagger$*]{Cian Jameson}
\author[$\ddagger$**]{John Sheekey}
\affil[$\ddagger$]{School of Mathematics and Statistics, University College Dublin, Belfield, Dublin 4, Ireland}
\affil[*]{\tt cian.jameson@ucdconnect.ie}
\affil[**]{\tt john.sheekey@ucd.ie}

\maketitle

\begin{abstract}
In this paper we completely classify spreads of $2$-dimensional subspaces of a $6$-dimensional vector space over a finite field of characteristic not two or three upon which a cyclic group acts transitively. This addresses one of the remaining open cases in the classification of flag-transitive linear spaces. We utilise the polynomial approach innovated by Pauley and Bamberg to obtain our results.   
\end{abstract}

\section{Introduction}

In this paper we aim to construct and classify {\it spreads} of a vector space upon which a cyclic group of automorphisms acts transitively, This corresponds to a classification of certain {\it flag-transitive linear spaces} with a prescribed automorphism group. The problem of classifying flag-transitive linear spaces has a long history, with a series of celebrated results culminating in \cite{BDDKLS} which classified most cases, leaving open the case of linear spaces arising from $t$-spreads of $V(tm,q)$ upon which a subgroup of $\GammaL(1,q^{tm})$ acts transitively. 

However this remaining open case remains a very difficult problem. In \cite{BP}, Bamberg and Pauley used a polynomial approach to give a new means of attacking this problem in the specific case of a cyclic group acting transitively on a $2$-spread in $V(2m,q)$, including constructing new examples. Recently in \cite{FL}, Feng and Lu used this approach and some results from permutation polynomials in order to find further examples.

In this paper we completely solve the case of $2$-spreads in a $6$-dimensional vector space over any finite field of characteristic not two or three. In particular we construct all possible examples, count the number of equivalence classes, and give canonical representatives for each equivalence class.

\section{Definitions and background}

Throughout the paper we let $q$ be a power of a prime $p>3$,  $\Fq$ the field with $q$ elements, and $\overline{\Fq}$ its algebraic closure. We denote by $V(n,q)$ a vector space of dimension $n$ over $\Fq$. We will use $\langle \rangle$ to denote the $\Fq$-span of a set or list of vectors or elements of an extension field of $\Fq$. 

\subsection{Spreads}

A {\it $t$-spread} in a vector space $V=V(n,q)$ is a set $\S$ of $t$-dimensional subspaces such that every nonzero vector of $V$ is contained in precisely one element of $\S$. A well-known result of Segre \cite{segre} tells us that a $t$-spread exists in $\Fq^n$ if and only if $n=tm$ for some positive integer $m$. The ``only if'' part of this statement follows by counting, while the ``if'' part follows from the so-called {\it Desarguesian spread}; if we identify $\FF_{q^{tm}}$ and $V(tm,q)$ as $\Fq$-vector spaces, then the set
\[
\D = \{\langle ax:x\in \FF_{q^t}\rangle:a\in \FF_{q^{tm}}^\times\}
\]
is a Desarguesian spread. 

We say that two $t$-spreads $\S_1$ and $\S_2$ are \textit{equivalent} (resp. \textit{projectively equivalent}) if there is an element of $\GammaL(n,q)$ (resp. $\GL(n,q)$) mapping $S_1$ to $S_2$. The {\it automorphism group} of a spread $\S$ is defined as the setwise stabiliser of $\S$ in $\GammaL(tm,q)$, and is denoted by $\Aut(\S)$. It is well known that the automorphism group of the Desarguesian spread is isomorphic to $\GammaL(m,q^t)$. Furthermore this group acts transitively on $\D$; in fact, it acts transitively on any set of $m+1$ elements of $\D$ in general position, where {\it general position} means that any $m$ elements of the set span all of $V$.

Note that we could equally work in the projective space $\PG(V)=\PG(tm-1,q)$. In this case for the above we would speak of a $(t-1)$-spread in an $(tm-1)$-dimensional projective space, and consider automorphisms of the spread as elements of $\PGammaL(tm,q)\simeq \GammaL(tm,q)/\Fq^\times$. As there is no consensus in the literature regarding whether to use a vector space or projective space setting, we choose to work with the former for convenience but may borrow terminology from the latter. In particular, we will consider $2$-spreads in $V(2m,q)$, but refer to them as {\it line spreads} when convenient.

\subsection{Linear spaces}
A {\it linear space} is a point-line incidence geometry $\cI$ in which
\begin{itemize}
\item[(i)] every pair of points is contained in precisely one common line;
\item[(ii)] every pair of lines meet in at most one common point. 
\end{itemize}
If every pair of lines meet in precisely one common point, it is called a {\it projective plane}. If for any line $\ell$ and any point $p$ not contained in  $\ell$ there exists a unique line containing $p$ and disjoint from $\ell$, it is called an {\it affine space}. 

A {\it flag} of a point-line incidence geometry is a pair $(p,\ell)\in \cP\times \cL$ such that $p\in \ell$. If a point $p$ is not contained in a line $\ell$ then $(p,\ell)$ is called an {\it anti-flag}.

Let $\cP$ and $\cL$ denote the set of points and lines of $\cI$ respectively. A bijective map $\phi$ from $\cP$ to itself is said to be an {\it automorphism} of $\cI$ if the image of the set of points on any line is again the set of points of a line. We denote the group consisting of all automorphisms of $\cI$ as $\Aut(\cI)$ and refer to it as {\it the (full) automorphism group} of $\cI$. We refer to any subgroup of $\Aut(\cI)$ as {\it a group of automorphism of $\cI$}.

We say that a linear space $\cI$ is {\it point-transitive} resp. {\it line-transitive} resp. {\it flag-transitive} if it possesses a group of automorphisms acting transitively on points resp. lines resp. flags. Much work has been done on classifying linear spaces with certain transitivity properties. We refer to \cite{BDDKLS} for an overview, and summarise the results relevant to this paper in the next section. 

\subsection{Linear spaces from spreads}
From a spread $\S$ of a vector space $V$ we can define a point-line incidence structure $\cI(\S)$ whose points are the elements of $V$ and whose lines are cosets of elements of $\S$; that is, cosets $u+U$ for $u\in V$ and $U\in \S$. It is straightforward to verify that $\cI(\S)$ satisfies the axioms of a linear space \cite{BarCof}; indeed, it has the further property of possessing {\it parallelism}. Such spaces are sometimes referred to as {\it translation Sperner spaces}. The lines through the point $u\in V$ are those of the form $u+U$ for $U\in \S$, and any vector $v\ne u$ is contained in $u+U$ if and only if $u-v\in U$. Since $\S$ is a spread, there is a unique spread element $U$ containing $u-v$. 

It is known that the automorphism group of the linear space $\cI(\S)$ is equal to $T.\Aut(\S)$, where $T$ denotes the group of {\it translations} (maps of the form $t_u:v\mapsto v+u$ for $u\in V$). The subgroup $T$ clearly acts transitively on points of $\cI(\S)$. Then any subgroup of automorphisms which acts transitively on flags of $\cI(\S)$ must be of the form $T.G$, where $G$ is a subgroup of $\Aut(\S)$ acting transitively on $\S$. Note that $\Aut(\S)$ acts transitively on $\S$ if and only if $\overline{\Aut}(\S)$ acts transitively on the induced spread of the projective space, and so for the purposes of studying flag-transitivity, it does not matter whether we consider spreads of a vector space or of the corresponding projective space.

In a series of seminal papers \cite{HigmanMcLaughlin, Buekenhoutetal,Liebeck,Saxl}, most cases were completely classified. 

\begin{theorem}
In order to classify all linear spaces with a flag-transitive automorphism group $H$, it remains only to classify the case  $H=TG_0$, where $T\cong(\Fqn,+)$ is a group of translations and $G_0\leq \GammaL(1,q^n)$.
\end{theorem}

For the remaining case of linear spaces with automorphism group contained in $\AGammaL(1,q^n)$, full classification remains open. Various constructions were provided by Kantor in \cite{Kantor}, leading him to suspect that a full classification may not be feasible. Hence additional restrictions on the linear space and the automorphism group are necessary in order to make headway towards classification; in particular, we seek to classify all $t$-spreads in $V(tm,q)$ possessing a transitive group of automorphisms $G$ contained in $\GammaL(1,q^{tm})$, regarded as a subgroup of $\GammaL(tm,q)$ in the natural way.

In \cite{BP} the authors considered the case of $t=2$ and $G$ a cyclic subgroup of $\GL(1,q^{2m})\simeq \FF_{q^{2m}}^\times$. In this paper we aim to utilise the techniques developed therein in order to further the constructions and classifications in this case, with particular focus on the case $m=3$. In this case the associated linear spaces possess $q^6$ points, with each line containing $q^2$ points.

\subsection{Transitive $2$-spreads}

For the remainder of this paper we will work with $2$-spreads of $V(2m,q)$, which one may also view as a line spread in $\PG(2m-1,q)$. We again identify $V(2m,q)$ with the elements of $\FF_{q^{2m}}$. We consider $2$-spreads  whose automorphism group contains the following group $C\leq \GL(1,q^{2m})\leq \GammaL(1,q^{2m})$:

\[
C:= \left\{x\mapsto cx:c^{\frac{(q-1)(q^{2m}-1)}{(q^2-1)}} = 1\right\}.
\]

Note that elements of $\GammaL(1,q^{2m})$ are of the form $x\mapsto ax^\sigma$ for some $
\sigma\in\Aut(\FF_{q^{2m}})$. Suppose $\S$ is a $2$-spread in $V(2m,q)$ on which the group $C$ acts transitively. Then $\S = \ell^C$ for some two-dimensional subspace $\ell$ of $V(2m,q)$. Since $C$ is normal in $\GammaL(1,q^{2m})$, it follows that for any $\phi\in \GammaL(1,q^{2m})$ we have $\phi(\ell^C)=\phi(\ell)^C$, and so $\ell^C$ and $\phi(\ell)^C$ are equivalent.

It can be shown that $\ell$ can be mapped by an element of $\GammaL(1,q^{2m})$ to a subspace of the form $\ell_\varepsilon$ for some $\varepsilon \in \Fqtm$, where $\ell_\varepsilon = \langle x -\varepsilon x^q :x\in \Fqt\rangle$. Thus it suffices to determine when $\ell_\epsilon^C$ is a $2$-spread. In \cite{BP}, these were characterised as follows.

\begin{theorem}\cite[Theorem 1]{BP}\label{thm:Pauleynomial}
A $2$-spread in $V(2m,q)$ upon which the group $C$ acts transitively is equivalent to one of the form $\ell_\varepsilon^C$, where $\varepsilon$ is an element of $\FF_{q^{2m}}$, and  
\[
\ell_\varepsilon = \langle x -\varepsilon x^q :x\in \Fqt\rangle.
\]
Moreover if $P(x)$ is the minimal polynomial of $\varepsilon$ over $\Fqt$, $\deg(P)=d$ and $\varepsilon^{q+1}\ne 1$, then $\ell_\varepsilon^C$ is a $2$-spread if and only if for all nonzero $x,y\in \Fqt$ it holds that 
\begin{equation}
\left(\frac{x^dP(x^{q-1})}{y^dP(y^{q-1})}\right)^{m/d}\in \Fq \implies \frac{x}{y}\in \Fq.\tag{{\bf Condition (1)}}
\end{equation}
\end{theorem}

\begin{theorem}\cite[Proposition 2]{BP}\label{thm:sprequiv}
Two $2$-spreads $\ell_\varepsilon^C$ and $\ell_\zeta^C$ of $V(2m,q)$ are equivalent if and only if
\[
\zeta^\sigma = \frac{v+u^q \varepsilon}{u+v^q \varepsilon}
\]
for some $u,v\in \Fqt$ with $u^{q+1}\ne v^{q+1}$, and some $\sigma \in \Aut(\Fqt:\Fq)$.
\end{theorem}

A straightforward simplification of this theorem gives that $\ell_\varepsilon^C$ and $\ell_\zeta^C$ are {\it projectively} equivalent if and only if $
\zeta = \frac{v+u^q \varepsilon}{u+v^q \varepsilon}$
for some $u,v\in \Fqt$ with $u^{q+1}\ne v^{q+1}$; that is, when we require that $\sigma$ is the identity automorphism.

\begin{definition}
For an irreducible polynomial $P(x)$ satisfying Condition (1), we will refer to a $2$-spread $\ell_\varepsilon^C$ defined by a root $\varepsilon$ of $P(x)$ as the {\it $2$-spread defined by $P(x)$}. If $P(x)$ and $Q(x)$ define (projectively) equivalent $2$-spreads then we will say that $P(x)$ and $Q(x)$ are {\it (projectively) equivalent}.
\end{definition}

Given this definition, the following follows immediately from Theorem \ref{thm:sprequiv}.

\begin{corollary}\label{cor:polequiv}
Two irreducible degree $d$ polynomials $P(x)$ and $Q(x)$ satisfying Condition (1) are equivalent if and only if
\[
Q(x) = \lambda (u+v^q x)^d P^\sigma \left( \frac{v+u^q x}{u+v^q x}\right)
\]
for some $\lambda,u,v\in \Fqt$ with $\lambda\ne 0,u^{q+1}\ne v^{q+1}$, and some $\sigma \in \Aut(\Fqt:\Fq)$.
\end{corollary}

Again the corresponding statement for projective equivalence can be obtained by omiting the automorphsism $\sigma$.

Note that this equivalence corresponds to equivalence under certain {\it linear fractional transformations} (often also called {\it M\"obius transformations}), namely those defined by the group generated by the following subgroup of $\GL(2,q^2)$, and field automorphisms.

\begin{definition}
    We denote by $U$ the subgroup of $\GL(2,q^2)$  defined as
    \[
    U :=\left\{\phi_{u,v} := \npmatrix{u^q&v\\v^q&u}:u,v\in \Fqt,u^{q+1} \ne v^{q+1}\right\}.
    \]
\end{definition}

Note that $U$ is isomorphic to $\GL(2,q)$. In fact, it is equal to the group of invertible {\it autocirculant matrices}, also known as {\it Dickson matrices}, in $\GL(2,q^2)$.

\subsection{Known examples}

We briefly summarise the known examples, with particular regard to the case of cubic polynomials, since these will be the main focus of this paper.

In \cite{BP} it was shown that the polynomial
\[
\mathrm{BP}_p(x) := \frac{x^{p+1}-1}{x-1}-2\in \FF_p[x]
\]
is irreducible and satisfies Condition (1). The only cubic polynomial in this family is the polynomial $x^3+x^2+x-1\in \FF_3[x]$. Since in this paper we consider only fields with characteristic greater than three, this example will not appear.

In \cite{Kantor}, various examples of transitive $2$-spreads were constructed. In \cite{BP}, it was shown that the only ones amongst these which arise from a $2$-spread with a transitive cyclic group of automorphisms are those of {\it Type 4}, which correspond to binomials, namely polynomials of the form 
\[
B_\theta(x) := x^n - \theta,
\]
where $\theta$ is a primitive element of $\Fqt$. We will study the general case of binomials in Section \ref{sec:binomials}. This family contains irreducible cubics satisfying Condition (1) if and only if $q\equiv 1\mod 3$, since no cubic binomial can be irreducible unless $q\equiv 1\mod 3$.

In \cite{FL}, Feng and Lu showed that the polynomials
\[g_{n,\rho}(x) = \frac{(\rho x -1)^n -\rho(x-\rho)^n}{\rho^n -\rho} \in \Fq[x],\]
where $\rho \in \Fqt^*$ has order $q+1$ and $n = d^tu$ for any odd divisor $d>1$ of $q+1$, any proper divisor $u$ of $d$ and any $t \in \NN^+$, have degree $n$, are irreducible in $\Fqt[x]$, and satisfy Condition (1). For the case $n=3$, we must have $d=3$ and $t=u=1$, and so $q \equiv 2 \mod 3$. Hence the cubics in this family are those of the form
\[
g_{3,\rho}(x) = x^3 -3x+(\rho+\rho^q),
\]
where $\rho$ has order $q+1$.

\section{A curve formulation}

We now show an equivalence between Condition (1) and properties of a curve $H_P$ related to $P(x)$. We introduce some notation which will be of use throughout.

\begin{definition}
Given a polynomial $P(x) = \sum_{i=0}^m a_i x^i \in \Fqt[x]$, we define
\begin{align*}
\tilde{P}(x) &:= \sum_{i=0}^m a_{m-i}^q x^i\\
G_P(z,w) &:= P(z)\tilde{P}(w)-\tilde{P}(z)P(w),\\
H_P(z,w) &:= \frac{P(z)\tilde{P}(w)-\tilde{P}(z)P(w)}{z-w}.
\end{align*}
\end{definition}

We will be concerned with zeroes of these polynomials of a certain form. We introduce the following set for convenience:
\[
Z:= \{(z,w)\in\Fqt^2:z^{q+1}=w^{q+1}=1,z\ne w\}.
\]

\begin{lemma}\label{lem:Gpol}
An irreducible polynomial $P(x)\in \Fqt[x]$ of degree $d=m$ satisfies Condition (1) if and only if $G_P$ has no zeroes in $Z$.
\end{lemma}

\begin{proof}
First we note that for any nonzero elements $a,b\in \overline{\Fq}$, we have that $a/b\in \Fq$ if and only if $ab^q-a^qb=0$, if and only if $a^{q-1}=b^{q-1}$. Applying this to the expressions from Theorem \ref{thm:Pauleynomial} we get that 
\[
\frac{x^mP(x^{q-1})}{y^mP(y^{q-1})}\in \Fq \Leftrightarrow x^{mq}P(x^{q-1})^q y^mP(y^{q-1}) = x^{m}P(x^{q-1}) y^{mq}P(y^{q-1})^q 
\]
for all nonzero $x,y\in \Fqt$. Now we define $z=x^{q-1},w=y^{q-1}$, and divide both sides by $(xy)^m$ to get
\[
\frac{x^mP(x^{q-1})}{y^mP(y^{q-1})}\in \Fq \Leftrightarrow z^{m}P(z)^q P(w) = P(z) w^{m}P(w)^q.
\]
Now observe that $z^m P(z)^q=\tilde{P}(z)$ and $w^m P(w)^q=\tilde{P}(w)$. Now $x/y\in \Fq$ if and only if $z=w$, and $z$ is a $(q-1)$-st power of a nonzero element of $\Fqt$ if and only if $z^{q+1}=1$. Thus Theorem \ref{thm:Pauleynomial} is equivalent to the claim. 
\end{proof}

As $G_P(z,w)$ is clearly divisible by $z-w$, and as dividing by $z-w$ does not affect the conditions, the following result in terms of $H_P(z,w)$ follows immediately.

\begin{lemma}\label{lem:Hpol}
An irreducible polynomial $P(x)\in \Fqt[x]$ of degree $d=m$ satisfies Condition (1) if and only if $H_P$ has no zeroes in $Z$.
\end{lemma}

\subsection{Two connections to permutation polynomials}\label{sec:permutataion}

A polynomial $f(x) \in \Fq[x]$ is called a \textit{permutation polynomial} of $\Fq$ if the map $x \mapsto f(x)$ is a permutation of $\Fq$. In \cite{FL}, the following connection between certain permutation polynomials and polynomials satisfying Condition (1) was shown.

\begin{lemma}\cite{FL}
Suppose $P(x)$ is a polynomial of degree $d$, where $\gcd(d,q-1)=1$. Then $x^dP(x^{q-1})$ is a permutation polynomial of $\Fqt$ if and only if $P(x)$ satisfies Condition (1).
\end{lemma}

Note however that this correspondence is only valid when $\gcd(d,q-1)=1$; when $\gcd(d,q-1) > 1$, a polynomial of the form $x^dP(x^{q-1})$ can never be a permutation polynomial, whereas there do exist polynomials satisfying Condition (1) in this case.

In \cite{bartolitimpanella}, permutation polynomials of $\Fqt$ of the form
\[f_{a,b}(X) = X(1 +aX^{q(q-1)} +bX^{2(q-1)}) \in \Fqt[X],\]
where $a,b \in \Fqt^*$, were completely characterized for finite fields with characteristic greater than 3. To attain their results, the authors consider the {\it algebraic plane curve} $\mathcal{C}_{a,b}$ with affine equation
\[F_{a,b}(X,Y) = \frac{(a^qX^3 +X^2 +b^q)(bY^3 +Y +a) -(a^qY^3 +Y^2 +b^q)(bX^3 +X +a)}{X-Y} =0.\]
It was shown that $f_{a,b}$ is a permutation polynomial of $\Fqt$ if and only if there is no point in $Z$ on $\mathcal{C}_{a,b}$. We observe that
\[F_{a,b}(X,Y) = -b^{q+1}H_P(X,Y)\]
where $P(x) = x^3 +b^{-1}x +ab^{-1}$. Hence we have the following.

\begin{lemma}\label{lem:bartoli}
Let $P(x)=x^3 +b^{-1}x +ab^{-1}$ for $a,b\in \Fqt$, $b\ne 0$. Then $f_{a,b}(x)$ is a permutation polynomial of $\Fqt$ if and only if $P(x)$ satisfies Condition (1).
\end{lemma}

Note however that it is not necessary for $P(x)$ to be irreducible in order for $f_{a,b}(X)$ to be a permutation polynomial, whereas it is required in order for $P(x)$ to define a cyclic spread.

From the results of \cite{bartolitimpanella}, we get full characterisation of cubics satisfying Condition (1) whose coefficient of $x^2$ is zero. However, we can not necessarily assume this, since not every cubic polynomial is equivalent under $U$ to one with this property. Hence this result is not sufficient to characterise all cubics satsifying Condition (1). Furthermore, \cite{bartolitimpanella} does not consider any question of equivalence, and indeed the notion of equivalence of cubic polynomials does not directly correspond to an equivalence amongst permutation polynomials of the form $f_{a.b}(x)$.

\subsection{Determining the reducibility of $H_P$}
In \cite{bartolitimpanella}, the authors show that for $q$ sufficiently large, if the curve $\mathcal{C}_{a,b}$ is absolutely irreducible then it must have points in $Z$. This was achieved by an application of the Aubry-Perret bound \cite{AubryPerret}. We will follow this method to generalise the result to the larger family of curves $\mathcal{H}_P$ with affine equation $H_P(X,Y)=0$ for arbitrary degree.

\begin{lemma}\label{lem:hpredlargeq}
Let $P(x) \in \Fqt[x]$ have degree $m$ and let $q$ be sufficiently large with respect to $m$. If the polynomial $H_P(z,w)$ is absolutely irreducible and not identically zero, then it has zeroes in $Z$ and hence $P$ does not satisfy Condition (1).
\end{lemma}

\begin{proof}
First let $e \in \Fqt \setminus \Fq$ such that $e^q = -e$, and define two transformations as in \cite{bartolitimpanella} by
\[\psi(X,Y) = \left(\frac{X+e}{X-e},\frac{Y+e}{Y-e}\right)\]
and
\[\phi(X,Y) = \left(e\frac{X+1}{X-1},e\frac{Y+1}{Y-1}\right).\]
Then the curve $\mathcal{H}^*_P$ defined by $K_P(X,Y) = (X-e)^{m-1}(Y-e)^{m-1} H_P(\psi(X,Y))$ and the curve $\mathcal{H}_P$ are $\Fqt$-isomorphic since $(X-1)^{m-1}(Y-1)^{m-1} K_P(\phi(X,Y)) = (2e)^{2(m-1)}H_P(X,Y)$. Note that $K_P(X,Y) \in \Fq[X,Y]$.

Let $\partial$ denote the degree of $K_P(X,Y)$ and $D$ the number of ideal points (i.e. points at infinity) of $\mathcal{H}^*_P$. By the Aubry-Perret bound \cite[Corollary 2.5]{AubryPerret}, the curve has affine $\Fq$-rational points $(x,y)$ with $x \ne y$ provided
\begin{align*}
& q+1 -(\partial-1)(\partial-2)\sqrt{q} -\partial -D >0\\
\iff & q > \frac{\left((\partial-1)(\partial-2)+\sqrt{\partial^4 -6\partial^3 +13\partial^2 -8\partial +4D}\right)^2}{4}. \tag{$\dagger$}
\end{align*}
Since $D \leq \partial \leq 2(m-1)$, $\mathcal{H}^*_P$ will have affine $\Fq$-rational points $(x,y)$ with $x \ne y$ if
\[q > \left((m-2)(2m-3)+\sqrt{(m-1)(4m^3 -24m^2 +49m -31)}\right)^2.\]
Thus for such $q$, there exists a point $\left(\frac{x+e}{x-e},\frac{y+e}{y-e}\right) \in Z$ that lies on $\mathcal{H}_P$. Therefore there are no degree $m$ polynomials $P$ satisfying Condition (1) for which $H_P$ is absolutely irreducible when $q$ satisfies the above inequality.
\end{proof}

Note that while Lemma 3.2 of \cite{AubryPerret} may appear to be more directly relevant to the curves considered here, we use instead Corollary 2.5 due to the fact that we will later have more information on the number $D$, leading to better bounds. 

\subsection{Preliminary restrictions on the factorisation of $H_P$}

Our strategy for the remainder of the paper will be to consider the possible factorisations of $H_P$. We begin by ruling out certain factors.

\begin{lemma}\label{P divides G}
Let $P(x) \in \F_{q^2}[x]$. Then $P(x)$ and $\tilde{P}(x)$ each divide both  $G_P(x^{q^2},x)$ and $H_P(x^{q^2},x)$.
\end{lemma}

\begin{proof}
We directly calculate that
\begin{align*}G_P(x^{q^2},x) &= P(x^{q^2})\tilde{P}(x) -\tilde{P}(x^{q^2})P(x)\\
&=P(x)^{q^2}\tilde{P}(x) -\tilde{P}(x)^{q^2}P(x)\\
&=P(x)\tilde{P}(x)[P(x)^{q^2 -1} -\tilde{P}(x)^{q^2 -1}],
\end{align*}
proving the first claim. 

Now $P(x)$ and $\tilde{P}(x)$ divide $G_P(x^{q^2},x) = (x^{q^2} -x)H_P(x^{q^2},x)$,
but do not divide $x^{q^2} -x$ (as otherwise a root $\varepsilon$ of either polynomial would satisfy $\varepsilon^{q^2} =\varepsilon$), they must divide $H_P(x^{q^2},x)$.
\end{proof}

\begin{lemma}\label{lem:a=b factor}
Let $P(x) \in \Fqt[x]$ be an irreducible polynomial of degree $m$. Then $H_P(z,w)$ cannot factorize as
\[\prod_{i=1}^{2(m-1)} (c_izw +a_i(z+w) +d_i)\]
for any $a_i,c_i,d_i \in \overline\Fq$.
\end{lemma}

\begin{proof}
Suppose that $H_P(z,w)$ factorizes as
\[\prod_{i=1}^{2(m-1)} (c_izw +a_i(z+w) +d_i)\]
for some $a_i,c_i,d_i \in \overline\Fq$ and let $\{\varepsilon^{q^{2i}} : 1 \leq i \leq m\}$ be the roots of $P$. Since $P(x)$ divides $H_P(x^{q^2},x)$, it must divide $cx^{q^2 +1} +a(x^{q^2} +x) +d$ for some $a,c,d \in \overline\Fq$. Thus
\begin{align*}
& c\left(\varepsilon^{q^{2(m-1)}}\right)^{q^2 +1} +a\left(\left(\varepsilon^{q^{2(m-1)}}\right)^{q^2} +\varepsilon^{q^{2(m-1)}}\right) +d = 0\\
\iff & c\left(\varepsilon^{q^{2(m-1)} +1}\right) +a\left(\varepsilon +\varepsilon^{q^{2(m-1)}}\right) +d = 0\\
\iff & c\left(\varepsilon^{q^{2(m-1)} +1}\right) +a\varepsilon^{q^{2(m-1)}} -(c\varepsilon^{q^2 +1} +a\varepsilon^{q^2}) =0\\
\iff & (\varepsilon^{q^{2(m-1)}}-\varepsilon^{q^2})(c\varepsilon +a)=0.
\end{align*}
If $\varepsilon^{q^2} = \varepsilon^{q^{2(m-1)}}$, then $\varepsilon = \varepsilon^{q^{2(m-2)}}$ which cannot occur because the smallest field containing $\varepsilon$ is $\F_{q^{2m}}$, so $a = -c\varepsilon$. Then
\[
c\varepsilon^{q^2 +1} +a(\varepsilon^{q^2} +\varepsilon) +d = 0 \iff d = c\varepsilon^2.
\]
Hence $P(x)$ divides
\[cx^{q^2 +1} -c\varepsilon(x^{q^2} +x) +c\varepsilon^2 = c(x^{q^2} -\varepsilon)(x-\varepsilon).\]
Since $P(x)$ cannot divide the linear factor, it must divide $x^{q^2} -\varepsilon$, which gives $\varepsilon = \varepsilon^{q^2}$. This contradiction means that $H_P(z,w)$ cannot factorize in this way.
\end{proof}

\section{Cubic polynomials}\label{sec:cubic}
We now focus on the case $m=3$, studying irreducible cubics in $\Fqt[x]$ satisfying Condition (1), and hence cyclic $2$-spreads in $V(6,q)$.

When $m=3$, we have that
\begin{align*}
-H_P(z,w) &= (\theta^q\delta+\gamma^q)z^2w^2+(\theta^q\gamma+\delta^q)(z^2w+zw^2)+ (\theta^{q+1}-1)(z^2+zw+w^2)\\
&\,\,\,+(\gamma^{q+1}-\delta^{q+1})zw+(\theta\gamma^q+\delta)(z+w)+(\theta\delta^q+\gamma)
\end{align*}
for $P(x) = x^3 -\delta x^2 -\gamma x -\theta \in \F_{q^2}[x]$. 

\subsection{Proving the reducibility of $H_P$}

In \cite{bartolitimpanella} it was shown via Lemma \ref{lem:bartoli} that when $\delta=0$, $P(x)$ can satisfy Condition (1) only if $H_P(z,w)$ is reducible. We use an identical approach to cover also the case when $\delta\ne 0$.

\begin{lemma}\label{lem:isreducible}
    Let $P(x) = x^3 -\delta x^2 -\gamma x -\theta \in \F_{q^2}[x]$. If $H_P$ is absolutely irreducible, then $P(x)$ does not satisfy Condition (1).
\end{lemma}

\begin{proof}
First suppose that $\theta^q\delta +\gamma^q\ne 0$, which ensures that $H_P(z,w)$ has degree four. We homogenise $H_P(z,w)$ to obtain the polynomial $\overline{H_P}(Z,W,X)$, obtaining $\overline{H_P}(Z,W,0) = -(\theta^q\delta +\gamma^q)Z^2W^2$. Hence $\mathcal{H}_P$ has precisely two ideal points. Applying inequality ($\dagger$) from the proof of Lemma \ref{lem:hpredlargeq} with $\partial=4$ and $D = 2$ yields that there are no cubic polynomials $P$ satisfying Condition (1) for which $H_P$ is absolutely irreducible when $q \geq 47$.

Finally suppose that $\theta^q\delta +\gamma^q= 0$, in which case we have $H_P(z,w)=(\theta^{q+1}-1)(\delta^qzw(z+w)-(z^2+zw+w^2)-\delta^{q+1}zw+\delta(z+w))$. If $\delta=0$, then $H_P=(1-\theta^{q+1})(z^2+zw+w^2)$, which is either identically zero or reducible. If $\delta\ne 0$, then $H_P$ has degree $3$, and homogenising we obtain $\overline{H_P}(Z,W,0) = (\theta^{q+1}-1)\delta^qZW(Z+W)$, and so there are three ideal points. Using again inequality ($\dagger$) with $\partial=3$ and $D = 3$ yields that there are no cubic polynomials $P$ satisfying Condition (1) for which $H_P$ is absolutely irreducible when $q \geq 13$.

For values of $q < 47$, an exhaustive Magma search returns that $H_P$ is reducible for any cubic $P$ satisfying Condition (1). 
\end{proof}

We now examine the case in which $H_P$ is reducible, and study the possible factorizations of $H_P$.

\subsection{Further restrictions on the factorization of $H_P$}
\begin{lemma}\label{lem:HPfactor}
Suppose $H_P(z,w)$ is reducible over $\overline{\Fq}$. Then $H_P(z,w)$ is reducible over $\Fqt$, and $H_P(z,w) = \mu(czw+az+bw+d)(czw+bz+aw+d)$ for some $a,b,c,d,\mu\in\Fqt$, where $a\ne b$.
\end{lemma}
\begin{proof}
Since $H_P(z,w)$ has degree at most 4, has degree at most $2$ in $z$ and in $w$,  and is symmetric in $z$ and $w$, we must have that either
\[H_P(z,w) = \mu(cz^2 +az +d)(cw^2 +aw +d) \tag{A}\]
or
\[H_P(z,w) = \mu(czw +az +bw +d)(czw +bz +aw +d) \tag{B}\]
or
\[H_P(z,w) = (czw +a(z+w) +d)(c'zw +b(z+w) +d') \tag{C}\]
for some $a,b,c,c',d,d',\mu \in \overline\Fq$.

By Lemma \ref{lem:a=b factor}, case (C) cannot occur and $a \ne b$ in case (B). Since the the coefficents of $H_P$ are in $\Fqt$, then raising the coefficients in the irreducible factors of $H_P$ must permute these factors up to scalar multiples. In case (A), we can assume without loss of generality that $a,c,d\in \Fqt$. In case (B) we can assume without loss of generality that $c,d\in \Fqt$, and either $a,b\in \Fqt$ or $a,b\in \FF_{q^4}$ with $a^{q^2}=b$.

If $H_P(z,w)$ factorizes as in (A), then by Lemma \ref{P divides G},
\begin{align*}
P(x) \mid H_P(x^{q^2},x) & =\mu(cx^{2q^2} +ax^{q^2}+d)(cx^2 +ax +d)\\
& =\mu(cx^2 +ax +d)^{q^2 +1}.
\end{align*}
As $P(x)$ is irreducible, it must divide $cx^2 +ax +d$. But the degree of $P(x)$ is 3, so case (A) cannot occur.

Thus $H_P(z,w)$ must factorize as in (B). If $a,b \not \in \Fqt$ then $b=a^{q^2}$ and
\[P(x) \mid H_P(x^{q^2},x) = \mu(cx^{q^2 +1} +ax^{q^2} +a^{q^2}x +d)(cx^{q^2 +1} +a^{q^2}x^{q^2} +ax +d).\]
Let $\varepsilon$ be a root of $P$. Then either
\[c\varepsilon^{q^2 +1} +(a\varepsilon)^{q^2} +a\varepsilon +d =0\]
or
\[c\varepsilon^{q^2 +1} +(b\varepsilon)^{q^2} +b\varepsilon +d =0.\]

We can assume without loss of generality that the first equation holds. Then raising both sides to the power of $q^2$ yields
\begin{align*}
c\varepsilon^{q^4 +q^2} +a\varepsilon^{q^4} +(a\varepsilon)^{q^2} +d &= 0 \\
\iff c\varepsilon^{q^4 +q^2} +a\varepsilon^{q^4} +(a\varepsilon)^{q^2} -(c\varepsilon^{q^2 +1} +(a\varepsilon)^{q^2} +a\varepsilon) &= 0\\
\iff (\varepsilon^{q^4}-\varepsilon)(c\varepsilon^{q^2} +a) &=0.
\end{align*}
The first factor cannot equal zero since $\Fqt(\varepsilon) = \F_{q^6}$. Hence $c\varepsilon^{q^2} =-a$. If $c=0$ then $a=0$, so $d=0$ and $H_P \equiv 0$. Thus $\varepsilon^{q^2} = -ac^{-1} \in \F_{q^4}$, which cannot occur since it is also a root of $P$. Hence $a,b \in \Fqt$.
\end{proof}

The following technical lemma will be of use in the subsequent theorem.
\begin{lemma}\label{lem:wquadratic}
Suppose $f(x)=ex^2+\lambda x+e^q$ for some $0\ne e\in \Fqt,\lambda\in\Fq$. Then $f(x)$ has a root $w$ such that $w^{q+1}=1$ if and only if its discriminant $\lambda^2-4e^{q+1}$ is either $0$ or a nonsquare in $\Fq$.
\end{lemma}

\begin{proof}
Let $w$ be a root of $f$. Then $w=\frac{-\lambda\pm\sqrt{\lambda^2-4e^{q+1}}}{2e}\in \Fqt$. Let $\Delta=\lambda^2-4e^{q+1}$, which is in $\Fq$. 

Suppose $\Delta$ is a square in $\Fq$. Then $(\sqrt{\Delta})^q=\sqrt{\Delta}$, and so
\begin{align*}
w^{q+1}&= \left(\frac{-\lambda\pm\sqrt{\Delta}}{2e}\right)\left(\frac{-\lambda\pm\sqrt{\Delta}}{2e^q}\right)   \\
&= \frac{\lambda^2+\Delta\mp 2\lambda \sqrt{\Delta}}{4e^{q+1}}\\
\end{align*}
Then $w^{q+1}=1$ if and only if $\lambda^2+\Delta\mp 2\lambda \sqrt{\Delta}=4e^{q+1}$, if and only if $2\Delta= \pm 2\lambda \sqrt{\Delta}$, if and only if $\Delta=0$ or $\lambda=\pm\sqrt{\Delta}$. But if $\lambda=\pm\sqrt{\Delta}$ then $e=0$, and so $w^{q+1}=1$ if and only if $\Delta=0$.

Suppose now that $\Delta$ is not a square in $\Fq$. Then $(\sqrt{\Delta})^q=-\sqrt{\Delta}$, and so
\begin{align*}
w^{q+1}&= \left(\frac{-\lambda\pm\sqrt{\Delta}}{2e}\right)\left(\frac{-\lambda\mp\sqrt{\Delta}}{2e^q}\right)   \\
&= \frac{\lambda^2-\Delta}{4e^{q+1}}\\
&=1,
\end{align*}
completing the proof.
\end{proof}

By Lemma \ref{lem:HPfactor}, we know the possible factorizations of $H_P$. We now find further restrictions on the possible values of $a,b,c,d$. Note that the roles of $a$ and $b$ are interchangeable, and so whenever we encounter a condition that must be satisfied by either $a$ or $b$, we can assume without loss of generality that it is satisfied by $a$.

\begin{lemma}\label{lem:ab=cd}
Suppose $H_P(z,w) = (czw +az +bw +d)(czw +bz +aw +d)$ for some $a,b,c,d \in \Fqt$, $a\ne b$. If $ab =cd$, then $P(x)$ is reducible.
\end{lemma}

\begin{proof}
First suppose $d \ne 0$. By Lemma \ref{P divides G}, $P(x)$ divides
\begin{align*}
G_P(x^{q^2},x) = (x^{q^2}-x)H_P(x^{q^2},x) &= \prod_{\lambda \in \Fqt}(x -\lambda)(cx^{q^2 +1} +ax^{q^2} +bx +d)(cx^{q^2 +1} +bx^{q^2} +ax +d)\\
&=\prod_{\lambda \in \Fqt}(x -\lambda)(x+ac^{-1})(x+bc^{-1})((cx+b)(cx+a))^{q^2}.
\end{align*}
Since $P(x)$ divides a product of linear factors with coefficients in $\Fqt$, it must be reducible. If $d=0$, then either $a=0$ or $b=0$. Suppose without loss of generality that $a=0$. Then $P(x)$ divides
\begin{align*}
G_P(x^{q^2},x) = (x^{q^2}-x)H_P(x^{q^2},x) &= \prod_{\lambda \in \Fqt}(x -\lambda)(cx^{q^2 +1} +bx)(cx^{q^2 +1} +bx^{q^2})\\
&=\prod_{\lambda \in \Fqt}(x -\lambda)(x(cx+b))^{q^2 +1},
\end{align*}
so $P(x)$ is again reducible.
\end{proof}

Hence when considering divisors of $H_P$, we can assume that $ab\ne cd$. We now find further conditions on the divisors of $H_P$ if $P$ satisfies Condition (1).

\begin{theorem}\label{thm:HPfacC1}
Let $H_\Psi(z,w) = czw+az+bw+d$, where $a,b,c,d\in \Fqt$, $ab \ne cd$. Then there exist $z,w\in \Fqt$ such that $H_\Psi(z,w)=0$, $w\ne z$, and $z^{q+1}=w^{q+1}=1$ if and only if 
\[
\Delta = (a^{q+1}-b^{q+1}+c^{q+1}-d^{q+1})^2-4(bd^q-a^qc)^{q+1},
\]
is zero or a nonsquare in $\Fq$, and the quadratic $(bd^q-a^qc)x^2+(d^{q+1}+b^{q+1}-c^{q+1}-a^{q+1})x+(b^qd-ac^q)$ possesses a root which is not a root of $cx^2+(a+b)x+d$.
\end{theorem}

\begin{proof}
Let $z,w\in\Fqt$ be such that $H_\Psi(z,w)=0$ and $z^{q+1}=w^{q+1}=1$. Then either $cw+a=bw+d=0$, or $z=-\left(\frac{bw+d}{cw+a}\right)$. In the first case we have $ab = -bcw = cd$, so by Lemma \ref{lem:ab=cd} $P(x)$ is reducible and does not satisfy Condition (1).

Next we suppose that $cw+a \ne 0$ and $z=-\left(\frac{bw+d}{cw+a}\right)$. Raising both sides to the power $q+1$, imposing $z^{q+1}=w^{q+1}=1$ and rearranging, we get that
\begin{equation}\label{eqn:wquad1}
(bd^q-a^qc)w^2+(d^{q+1}+b^{q+1}-c^{q+1}-a^{q+1})w+(b^qd-ac^q)=0.
\end{equation}
If $bd^q-a^qc\ne 0$, then this is a quadratic equation in $w$ with coefficients in $\Fqt$ satisfying the conditions of Lemma \ref{lem:wquadratic}. The discriminant of the quadratic is
\[
\Delta = (a^{q+1}-b^{q+1}+c^{q+1}-d^{q+1})^2-4(bd^q-a^qc)^{q+1},
\]
and so from Lemma \ref{lem:wquadratic} we have that $w^{q+1}=1$ if and only if $\Delta$ is either zero or a nonsquare in $\Fq$.

Now $z=w$ if and only if $w=-\left(\frac{bw+d}{cw+a}\right)$, if and only if $cw^2+(a+b)w+d=0$. Thus we have a solution with $z\ne w$ if and only if not every solution of equation (\ref{eqn:wquad1}) is also a solution of $cw^2+(a+b)w+d=0$. 
\end{proof}

We summarise the results of this section with the following statement.

\begin{corollary}
Let $P(x)$ be an irreducible cubic in $\Fqt[x]$. Suppose $H_P(z,w)$ is reducible, with $H_P(z,w) = (czw+az+bw+d)(czw+bz+aw+d)=0$ for some $a,b,c,d\in\Fqt$, and let $\Delta$ be as in Theorem \ref{thm:HPfacC1}. Then $P(x)$ satisfies Condition (1) if and only if one of the following occur:
\begin{itemize}
    \item $\Delta$ is a nonzero square in $\Fq$;
    \item $\Delta$ is a nonsquare in $\Fq$ and the quadratic polynomials $(bd^q-a^qc)x^2+(d^{q+1}+b^{q+1}-c^{q+1}-a^{q+1})x+(b^qd-ac^q)$ and $cx^2+(a+b)x+d$ are nonzero scalar multiples of each other;
    \item $\Delta=0$, $bd^q-a^qc\ne 0$, and the unique root of $(bd^q-a^qc)x^2+(d^{q+1}+b^{q+1}-c^{q+1}-a^{q+1})x+(b^qd-ac^q)$ is a root of $cx^2+(a+b)x+d$.
\end{itemize}
\end{corollary}

\section{Binomials}\label{sec:binomials}
In this section, we determine exact conditions for when a binomial satisfies Condition (1). Note that we will start in the case of a binomial of arbitrary degree, before stating the consequences for cubics.

\begin{lemma}\label{lem:binomialcond1}
Let $P(x)=x^m-\theta \in \Fqt[x]$, where $m>2$ is an integer. Then $H_P(z,w)$ is not identically zero and reducible if and only if $\theta^{q+1} \ne 1$. Furthermore, $P(x)$ satisfies Condition (1) if and only if $\gcd(m,q+1) = 1$.
\end{lemma}

\begin{proof}
We calculate that
\[G_P(z,w) = (\theta^{q+1}-1)(w^m -z^m).\]
Hence $G_P$ has a zero in $Z$ if and only if there exists $(z,w) \in \Fqt^2$ with $z^m=w^m$, $z^{q+1}=w^{q+1}=1$ and $z \ne w$. This occurs precisely when $\gcd(m,q+1) \ne 1$, so $P(x)$ satisfies Condition (1) if and only if $m$ and $q+1$ are coprime. Note that $G_P$ is identically zero if and only if $\theta^{q+1}=1$.
\end{proof}

We can apply the next well-known result to determine when $P(x)$ is irreducible.

\begin{lemma}\cite[Theorem 3.75]{lidlniederreiter}\label{lem:binomialirred} Let $m \geq 2$ be an integer and let $\theta \in \Fq^*$. Then $x^m -\theta \in \Fq[x]$ is irreducible if and only if the following hold:
\begin{enumerate}[(i)]
    \item $\operatorname{rad}(m) \mid o(\theta)$;
    \item $\gcd\left(m,\frac{q-1}{o(\theta)}\right) =1$;
    \item if $m \equiv 0 \mod 4$ then $q \equiv 1 \mod 4$.
\end{enumerate}
\end{lemma}

When $m=3$, we can combine Lemmas \ref{lem:binomialcond1}  and \ref{lem:binomialirred} to give the following.

\begin{theorem}\label{thm:binomialirredcond1}
    A cubic binomial $x^3 -\theta \in \Fqt[x]$ is irreducible and satisfies Condition (1) if and only if $q\equiv 1\mod 3$ and $3$ does not divide $\frac{q^2-1}{o(\theta)}$. 
\end{theorem}

\begin{proof}
Suppose $x^3-\theta$ is irreducible and satisfies Condition (1). By Lemma \ref{lem:binomialcond1} we must have $q\equiv 1\mod 3$, and by Lemma \ref{lem:binomialirred} we have that $3$ does not divide $\frac{q^2-1}{o(\theta)}$. Thus the two conditions are necessary.

Suppose now that $q\equiv 1\mod 3$ and $3$ does not divide $\frac{q^2-1}{o(\theta)}$. Then $o(\theta)$ does not divide $\frac{q^2-1}{3} = (q+1)\left(\frac{q-1}{3}\right)$, and since $\frac{q-1}{3}$ is an integer, we get that $\theta^{q+1}\ne 1$. Finally since $3$ divides $q^2-1$ we must have that $3$ divides $o(\theta)$, and so $x^3-\theta$ is irreducible and satisfies Condition (1), showing that the two conditions are sufficient.
\end{proof}

\begin{remark}
    The case of binomials $x^m-\theta$ with $\theta$ a primitive element of $\Fqt$ and $m$ an odd divisor of $q-1$ corresponds Kantor's Type 4 construction. Thus we have a generalisation of this family, both in terms of new inequivalent examples when $m$ divides $q-1$, and new values of $m$. For example, this section shows that there exist irreducible binomials of degree $25$ over $\FF_{11^2}$ satisfying Condition (1), and hence new $2$-spreads of $V(50,11)$ with a cyclic transitive group of automorphisms, and new flag-transitive linear spaces.
\end{remark}

\section{Characterisation of cubics}

We are now ready to fully characterise the  irreducible cubic polynomials satisfying Condition (1). We split them into three (not necessarily non-empty) parameterised families,

\begin{theorem}\label{thm:character}
Let $P(x)=x^3-\delta x^2-\gamma x-\theta \in \Fqt[x]$ be irreducible. Then $H_P(z,w)$ is not identically zero and reducible if and only if one of the following holds:
\begin{align*}
P(x) = B_{\theta}(x) & :=x^3-\theta, ~\theta^{q+1} \ne 1;\\
P(x) = P_{\delta,\alpha}(x) &:= x^3-\delta x^2-(\delta\alpha+3\alpha^{1-q})x- (\delta\alpha^2(1-\alpha^{-(q+1)})/3+\alpha^{2-q}), ~\alpha\ne 0;\\
P(x) = Q_{\delta,\gamma}(x) &:= x^3-\delta x^2-\gamma x+\delta\gamma/9, ~\gamma^{q+1}=9.
\end{align*}
Moreover, 
\begin{itemize}
\item an irreducible of the form $B_\theta(x)$ satisfies Condition (1) if and only if $\theta^{q+1}\ne 1$ and $q \equiv 1 \mod 3$;
\item an irreducible of the form $P_{\delta,\alpha}(x)$ satisfies Condition (1) if and only if $\frac{4-\alpha^{q+1}}{3\alpha^{q+1}}$ is a nonzero square in $\Fq$, and either $\delta=0$ or $(\alpha+3\delta^{-q})^{q+1}\ne 1$;
\item an irreducible of the form $Q_{\delta,\gamma}(x)$ satisfies Condition (1) if and only if $~\gamma^{\frac{q+1}{2}}=3$.
\end{itemize}
\end{theorem}

\begin{proof}
We first note that the set of polynomials $\{z^2w^2,z^2w+zw^2,z^2+zw+w^2,zw,z+w,1\}$ is linearly independent in $\Fqt[z,w]$. By Lemmas \ref{lem:isreducible} and \ref{lem:HPfactor} we have that  \[
H_P(z,w)=\mu(czw +az +bw +d)(czw +bz +aw +d)
\]
for some $a,b,c,d,\mu \in \Fqt$. Thus by comparing coefficients (see the beginning of Section \ref{sec:cubic}) we see that  
\[
\begin{array}{l|ll}
\textrm{(1A)}&-(\theta^q\delta+\gamma^q)&=\mu c^2\\
\textrm{(1B)}&-(\theta\delta^q+\gamma)&=\mu d^2\\
\textrm{(2A)}&-(\theta^q\gamma+\delta^q)&=\mu c(a+b)\\
\textrm{(2B)}&-(\theta\gamma^q+\delta)&=\mu d(a+b)\\
\textrm{(3)}&1 -\theta^{q+1}&=\mu ab\\
\textrm{(4)}&\delta^{q+1}-\gamma^{q+1}&=\mu(2cd+a^2+b^2-ab)\\
\end{array}
\]

{\bf Case 1:} Assume $c=0$. Then $(\theta\delta^q + \gamma)^q = \theta^q\delta+\gamma^q=0$, and so $d=0$. Therefore $\delta = -\theta\gamma^q$, so (1A) and (3) imply that $ab = 0$ or $\gamma = 0$. If either $a=0$ or $b=0$, (3) and (4) require that $a = b = 0$, giving $H_P(z,w) \equiv 0$. Thus $\gamma =0$, which implies that $\delta =0$ and so $P(x) = x^3 -\theta=B_\theta(x)$. The binomial case is characterised in Theorem \ref{thm:binomialirredcond1}.

{\bf Case 2:} Assume $c \neq 0$ and $a+b\ne 0$. We may assume without loss of generality that $c=1$. Since $(\theta^q \delta + \gamma^q)^q = \theta \delta^q +\gamma$, equations (1A) and (1B) tell us that $\mu^{q-1} =d^2$. Since $(\theta^q \gamma +\delta^q)^q = \theta \gamma^q +\delta$, equations (2A) and (2B) give that $\mu^q (a+b)^q = \mu d(a+b)$. Thus we have $a+b = d(a+b)^q$, so $d = (a+b)^{1-q}$ and the following equations hold:
\[
\begin{array}{l|ll}
\textrm{(1A)}&-(\theta^q\delta+\gamma^q)&=\mu\\
\textrm{(1B)}&-(\theta\delta^q+\gamma)&=\mu(a+b)^{2 -2q}\\
\textrm{(2A)}&-(\theta^q\gamma+\delta^q)&=\mu(a+b)\\
\textrm{(2B)}&-(\theta\gamma^q+\delta)&=\mu(a+b)^{2-q}\\
\textrm{(3)}&1 -\theta^{q+1}&=\mu ab\\
\textrm{(4)}&\delta^{q+1}-\gamma^{q+1}&=\mu(2(a+b)^{1-q}+a^2+b^2-ab)\\
\end{array}
\]
To obtain an expression for $\theta$ in terms of $\delta, a$ and $b$, we substitute the expression for $\gamma^q$ from (1A) into (2B) to yield $\mu(\theta -(a+b)^{2-q}) = \delta(\theta^{q+1} -1)$. Replacing $\theta^{q+1} -1$ using (3) and dividing by $\mu$ we get $\theta = \delta ab +(a+b)^{2-q}$. To obtain an expression for $\gamma$, we first multiply (1A) by $\gamma$, then substitute in the expression for $\theta^q \gamma$ from (2A) to get $\mu(\gamma -\delta(a+b)) = \gamma^{q+1} -\delta^{q+1}$. Replacing the right-hand side using (4) and dividing by $\mu$, we get $\gamma = \delta(a+b) +a^2 -ab +b^2 +2(a+b)^{1-q}$.

For convenience in the remaining calculations, we define $\alpha=a+b$, and $\beta=ab$. Note that we are assuming that $\alpha\ne 0$. Then our expressions for $\gamma$ and $\theta$ become
\begin{align*}
\gamma &= \delta\alpha +\alpha^2-3\beta +2\alpha^{1-q},\\
\theta &= \delta \beta +\alpha^{2-q}.
\end{align*}

We substitute these expressions into (1A),  obtaining
\[\mu = \alpha^{q-1}(2 +\delta^q\alpha +\delta\alpha^q +\alpha^{q+1}) +\beta^q(\delta^{q+1} -3)\]
and hence from (1B) we have that
\[\alpha^{2-2q}(\alpha^{2q -2}\beta -\beta^q)(\delta^{q+1} -3) = 0.\]
Suppose $\alpha^{2q -2}\beta - \beta^q\ne 0$. Then $\delta^{q+1} =3$. Equation (2A) says that
\[\delta^q(1 -\alpha^{q+1} +(\alpha^2 +2\alpha^{1-q})\beta^q -3\beta^{q+1}) = 3\alpha(\alpha^{2q -2}\beta -\beta^q),\]
so multiplying both sides by $\delta$ and rearranging gives
\[\delta = \frac{1 -\alpha^{q+1} +(\alpha^2 +2\alpha^{1-q})\beta^q -3\beta^{q+1}}{\alpha(\alpha^{2q -2}\beta -\beta^q)}=:\frac{X}{Y},\]
where $X$ denotes the displayed numerator and $Y$ the denominator. Then $\delta^{q+1} = 3 \iff X^{q+1} -3Y^{q+1} = 0$.

Observe that $X^q = X +(\alpha +2\alpha^{-q})Y$ and $Y^q = -\alpha^{1-q}Y$. Hence
\begin{align*}
0&=X^{q+1} -3Y^{q+1}\\
\iff 0&=X^2 +(\alpha +2\alpha^{-q})XY +3\alpha^{1-q}Y^2 \\
\iff 0&=\alpha^q X^2 +(\alpha^{q+1} +2)XY +3\alpha Y^2 \\
\iff 0&=\alpha^q\left(\frac{X}{Y}\right) +\alpha^{q+1} +2 +3\alpha \left(\frac{Y}{X}\right)\\
\iff 0&=\alpha^q\left(\frac{X}{Y}\right) +\alpha^{q+1} +2 +\left(\frac{X^{q+1}}{Y^{q+1}}\right)\alpha \left(\frac{Y}{X}\right) &\\
\iff 0&=2 +\delta^q\alpha +\delta\alpha^q +\alpha^{q+1},
\end{align*}
in which case $\mu=0$, which contradicts $H_p\not\equiv 0$.

Thus we must have $\alpha^{2q -2}\beta = \beta^q$, so $Y=0$. Equation (3) states that $X^q = \delta^q Y^q = 0$,
so $X=0$ also. Hence
\begin{align*}
0 = X &= 1 -\alpha^{q+1} +(\alpha^2 +2\alpha^{1-q})\beta^q -3\beta^{q+1}\\
&=1 -\alpha^{q+1} +(\alpha^2 +2\alpha^{1-q})\alpha^{2q -2}\beta -3\alpha^{2q -2}\beta^2\\
&= (\alpha^{q-1}\beta -1)(\alpha^{q+1} -3\alpha^{q-1}\beta -1).
\end{align*}
If $\beta = \alpha^{1-q}$, then $P(x)$ has $\delta +\alpha$ as a root and so is reducible. Thus we have
\begin{align*}
\alpha^{q+1} -3\alpha^{q-1}\beta = 1\\
\iff \alpha^{q-1}(\alpha^2 -3\beta) =1\\
\iff \alpha^2 -3\beta = \alpha^{1-q}.
\end{align*}
This yields the expressions for $\gamma$ and $\theta$ which gives $P(x)=P_{\delta,\alpha}$(x).\\

We note that without loss of generality, we may assume that
\begin{subequations}
\label{eqn:abalpha}
\begin{align}
    a &=\frac{\alpha}{2}\left(1+\sqrt{\frac{4-\alpha^{q+1}}{3\alpha^{q+1}}}\right), \label{eqn:abalpha1}\\
    b &= \frac{\alpha}{2}\left(1-\sqrt{\frac{4-\alpha^{q+1}}{3\alpha^{q+1}}}\right) \label{eqn:abalpha2}.
\end{align}
\end{subequations}

Now $H_{P_{\delta,\alpha}}\equiv 0$ if and only if $\delta^{q+1}\left(\frac{\alpha^{q+1}-1}{3}\right)+\delta^q\alpha +\delta\alpha^q+3=0$, which occurs if and only if $\delta\ne 0$ and $(\alpha+3\delta^{-q})^{q+1}=1$. In this case $P_{\delta,\alpha}(x)$ does not satisfy Condition (1).

If $H_{P_{\delta,\alpha}}\not\equiv 0$ and $bd^q-a^qc\ne 0$, then the quadratic $(bd^q-a^qc)x^2+(d^{q+1}+b^{q+1}-c^{q+1}-a^{q+1})x+(b^qd-ac^q)$ is a nonzero scalar multiple of the quadratic $cx^2+(a+b)x+d$, since $(bd^q-a^qc)(a+b)=b(a+b)^q-a^q(a+b)=b^{q+1}-a^{q+1}=d^{q+1}+b^{q+1}-c^{q+1}-a^{q+1}$, and $(bd^q-a^qc)(a+b)^{1-q}= b-a^q(a+b)^{1-q}=b+b^q(a+b)^{1-q}-(a+b)^q(a+b)^{1-q}=b^q(a+b)^{1-q}-a=b^qd-ac^q$, and so by Theorem \ref{thm:HPfacC1}, $P_{\delta,\alpha}(x)$ satisfies Condition (1).

Now if $bd^q-a^qc= 0$, then the first quadratic is identically zero, and so $P_{\delta,\alpha}(x)$ does not satisfy Condition (1). This occurs if and only if $a^{q+1}=b^{q+1}$, if and only if $\frac{4-\alpha^{q+1}}{3\alpha^{q+1}}$ is zero or a nonsquare in $\Fq$.

{\bf Case 3:} Assume $c \neq 0$ and $a+b= 0$. Again we assume without loss of generality that $H_P(z,w)$ factorises as
\[\mu(zw +az +bw +d)(zw +bz +aw +d)\]
for some $\mu \in \Fqt^*$. Then the following equations hold:
\[
\begin{array}{l|ll}
\textrm{(1A)}&-(\theta^q\delta+\gamma^q)&=\mu\\
\textrm{(1B)}&-(\theta\delta^q+\gamma)&=\mu d^2\\
\textrm{(2A)}&-(\theta^q\gamma+\delta^q)&=0\\
\textrm{(2B)}&-(\theta\gamma^q+\delta)&=0\\
\textrm{(3)}&1-\theta^{q+1}&=-\mu a^2\\
\textrm{(4)}&\delta^{q+1}-\gamma^{q+1}&=\mu(2d+3a^2)\\
\end{array}
\]
From (2B), we have $\delta = -\theta\gamma^q$. Substituting this into (1) gives $-\gamma^q(\theta^{q+1} -1) = \mu$ and so $\gamma^q a^2 =1$ by (3). Hence $a^2 = \gamma^{-q}$. Equation (1B) tells us that
\begin{align*}
\theta(-\theta\gamma^q)^q +\gamma &= \mu d^2\\
\iff -\gamma(\theta^{q+1} -1) &= \mu d^2\\
\iff \gamma a^2 &= d^2\\
\iff \gamma^{1-q} &= d^2.
\end{align*}

Substituting the expression for $\delta$ into (4) gives
\begin{align*}
\gamma^{q+1} -(-\theta\gamma^q)^{q+1} &= \mu (2d +3a^2)\\
\iff -\gamma^{q+1}(\theta^{q+1} -1) &= \mu (2d +3a^2)\\
\iff \gamma^{q+1}a^2 &= 2d +3a^2\\
\iff \gamma &= 2d +3\gamma^{-q}\\
\iff d &= \frac{\gamma -3\gamma^{-q}}{2}.
\end{align*}
Squaring the last equation yields
\begin{align*}
\gamma^{1-q} = d^2 &= \frac{\gamma^2 -6\gamma^{1-q} +9\gamma^{-2q}}{4}\\
\iff \gamma^2 -10\gamma^{1-q} +9\gamma^{-2q} &= 0\\
\iff \gamma^{2(q+1)} -10\gamma^{q+1} +9 &=0 \\
\iff \gamma^{q+1} = 1 \text{ or } \gamma^{q+1} &= 9.
\end{align*}

If $\gamma^{q+1} =1$ then $\gamma = \gamma^{-q}$, so (2B) states that
\begin{align*}
\theta\gamma^q &=-\delta\\
\iff \theta\gamma^{-1} & =-\delta\\
\iff \theta & = -\delta\gamma.
\end{align*}
The polynomial $P(x) = x^3 -\delta x^2 -\gamma x +\delta\gamma$ has $\delta$ as a root and is hence reducible, so we must have $\gamma^{q+1} =9$.\\

If $\gamma^{q+1} =9$ then $\gamma = 9\gamma^{-q}$, so $d^2 = \gamma^2/9$ and $d = \pm \gamma/3$. If $d = \gamma/3$, equations (1A)...(4) hold. If $d = -\gamma/3$, we arrive at a contradiction in (4) with $1 = -3$. We now have $P(x) = x^3 -\delta x^2 -\gamma x +\delta\gamma/9$, where $\gamma^{q+1} = 9$, $a=-b$ and $d = \gamma/3$.

By Theorem \ref{thm:HPfacC1}, there exist $z,w\in \Fqt$ such that $H_P(z,w)=0$ and $z^{q+1}=w^{q+1}=1$ if and only if
\begin{align*}
\Delta &= (a^{q+1}-b^{q+1}+c^{q+1}-d^{q+1})^2-4(bd^q-a^qc)^{q+1}\\
&=\left(\frac{-4}{27}\right)(\gamma^{(q+1)/2} +3)^2
\end{align*}
is zero or a nonsquare in $\mathbb{F}_q$. Since $\gamma^{q+1} =9$, $\gamma^{(q+1)/2} =\pm 3$. Hence
\[\Delta = \begin{cases}
    \frac{16}{-3} = \frac{4^2}{-3},& \text{if } \gamma^{(q+1)/2} = 3\\
    0,& \text{if } \gamma^{(q+1)/2} = -3.
\end{cases}\]
When $\gamma^{(q+1)/2} =-3$, the first quadratic in the statement of Theorem \ref{thm:HPfacC1} is identically zero, and so Condition (1) is never satisfied. When $\gamma^{(q+1)/2} =3$ and $q \equiv 2 \mod 3$, $\Delta$ is a nonzero nonsquare. The two quadratics in the statement of Theorem \ref{thm:HPfacC1} are 
$-(a\gamma^q/3+a^q)x^2-(a\gamma^q/3+a^q)^q$ and $x^2+\gamma/3$ respectively. These are scalar multiples of each other, since $(a\gamma^q/3+a^q)\gamma/3 = a+a^q\gamma/3 = (a\gamma^q/3 +a^q)^q$, and hence Condition (1) is always satisfied.

When $\gamma^{(q+1)/2} =3$ and $q \equiv 1 \mod 3$, $\Delta$ is a nonzero square, and hence Condition (1) is satisfied.
\end{proof}

\section{Classification of cubics}

In this section we determine the number and the nature of the equivalence classes of irreducible cubics satisfying Condition (1). We begin by enumerating the irreducible cubics satisfying Condition (1), and subsequently find representatives for each equivalence class.

\subsection{Enumeration}
We first need some technical lemmas which will enable us to perform the desired enumeration. To start, we introduce the following characterisation of irreducible cubic polynomials of Dickson \cite{Dickson}.

\begin{lemma}\label{lem:dickson}
The cubic $x^3 +sx +t = 0 \in \Fq[x]$ is irreducible over $\Fq$ if and only if the following two conditions hold:
\begin{itemize}
    \item $R := -4s^3 -27t^2$ is a nonzero square in $\Fq$; 
    \item $S := (-t +\mu\sqrt{-3})/2$ is a noncube in $\Fq\left(\sqrt{-3}\right)$, where $R = 81\mu^2$.
\end{itemize}
Moreover, if $R$ is a (not necessarily nonzero) square in $\Fq$, then this cubic has either zero or three roots in $\Fq$.
\end{lemma}

We apply this result to the polynomials $P_{\delta,\alpha}(x)$ to obtain useful criteria towards counting irreducible polynomials of this form satisfying Condition (1).

\begin{lemma}\label{lem:pdadisc}
A polynomial of the form $P_{\delta,\alpha}(x)\in \Fqt[x]$ is either irreducible or has all three of its roots in $\Fqt$. Furthermore, it is reducible if and only if at least one of the following holds:
\begin{itemize}
    \item $\alpha^{q+1} =4;$
    \item $\delta = \frac{-3\alpha}{2} \left(1 + \sqrt{1 -4\alpha^{-(q+1)}}\right);$
    \item $\delta = \frac{-3\alpha}{2} \left(1 +\frac{\kappa^3 +1}{\kappa^3 -1} \sqrt{1 -4\alpha^{-(q+1)}}\right),$
\end{itemize}
for some $\kappa \in \F_{q^2}$.
\end{lemma}

\begin{proof}
We first perform a change of variables in order to apply Lemma \ref{lem:dickson}. Let $x = y+\delta/3$. Then
$P_{\delta,\alpha}(x) = y^3 +sy +t$,
where
\[s = -(3\alpha^{1-q} +\delta\alpha +\delta^2/3); \quad
t = -(3\alpha +2\delta)(9\alpha^{1-q} +3\alpha\delta +\delta^2)/27.\]
Using the notation of Lemma \ref{lem:dickson},
\[R = \frac{-\alpha^{1-q}}{3}(\alpha^{q+1} -4)(9\alpha^{1-q} +3\alpha\delta +\delta^2)^2\]
Hence $R$ is always a square in $\Fqt$, and thus by Lemma \ref{lem:dickson} the first claim holds.

For convenience, define $r := \sqrt{1 -4\alpha^{-(q+1)}}.$ Then it is clear that $R$ is zero if and only if $\alpha^{q+1} =4$ or
\[\delta = \delta_{\pm} := -\frac{3\alpha}{2}(1 \pm r)\]
Now
\[S = \frac{(\delta -\delta_\pm)^2(\delta-\delta_\mp)}{27} = \frac{\delta -\delta_\mp}{\delta-\delta_\pm}\left(\frac{\delta -\delta_\pm}{3}\right)^3.\]
Hence $S$ is a cube if and only if
\[\frac{\delta-\delta_-}{\delta-\delta_+}\]
is a cube. Suppose $\frac{\delta-\delta_-}{\delta-\delta_+} = \kappa^3$ for some $\kappa \in \Fqt$. If $\kappa^3 =1$, then $r=0$ and so $\alpha^{q+1}=4$. If $\kappa^3 \ne 1$, then
\[\delta = \frac{(\delta_-) -(\delta_+)\kappa^3}{1-\kappa^3} = \frac{-3\alpha}{2}\left(1 +\frac{\kappa^3 +1}{\kappa^3 -1}r\right),\]
completing the proof.
\end{proof}

We saw in Theorem \ref{thm:character} that the case where $(\alpha+3\delta^{-q})^{q+1} =1$ appears to require special attention. We show now that in this case, a polynomial satisfying Condition (1) is reducible, and so can be disregarded.

\begin{lemma}\label{lem:pdared}
If $(\alpha +3\delta^{-q})^{q+1} =1$ and $(4\alpha^{-(q+1)}-1)/3$ is a nonzero square in $\F_q$ then $P_{\delta,\alpha}(x) \in \Fqt[x]$ is reducible.
\end{lemma}

\begin{proof}
Let $(4\alpha^{-(q+1)}-1)/3 = \lambda^2$ for some $\lambda \in \F^*_q$ and let $r = \sqrt{1-4\alpha^{-(q+1)}}$. Then $r = \sqrt{-3} \lambda \in \F_q \iff \sqrt{-3} \in \F_q \iff q \equiv 1 \mod 3$. We also note that $r \ne \pm 1$ since $\alpha \ne 0$. We claim that any $\delta$ satisfying $(\alpha +3\delta^{-q})^{q+1} =1$ is of the form listed in Lemma \ref{lem:pdadisc}. There are at most $q+1$ such $\delta$ when $\alpha^{q+1} \ne 1$ and at most $q$ otherwise. Define \[\delta_\kappa := \frac{-3\alpha}{2} \left(1 +\frac{\kappa^3 +1}{\kappa^3 -1} r\right),\]
where $\kappa \in \Fqt$ and $\kappa^3 \ne 1$.

We first suppose $q \equiv 1 \mod 3$. Then $(\alpha +3\delta_\kappa^{-q})^{q+1} =1 \iff \kappa^{3(q+1)}(r +1)^3 +(r-1)^3 =0$. For each $r$, there exist $q+1$ elements $\kappa \in \F_{q^2}$ such that
\[\kappa^{q+1} = \frac{1-r}{1+r}\]
since
\[\left(\kappa^{q+1}\right)^{q-1} = 1 = \left(\frac{1-r}{1+r}\right)^{q-1}.\]
Note that $\delta_\kappa = \delta_\iota$ if and only if $\kappa^3 = \iota^3$. Since $\kappa^{q+1}=\iota^{q+1}$ and $q \equiv 1 \mod 3$, the $q+1$ values of $\kappa$ such that $\kappa^{q+1} = \frac{1-r}{1+r}$ give $q+1$ distinct solutions $\delta = \delta_\kappa$ to $(a +3\delta^{-q})^{q+1}$, provided $\kappa^3 \ne 1$. If $\kappa^3 =1$, then
\[\frac{1-r}{1+r} = \kappa^{q+1} = \kappa^2(\kappa^3)^{(q-1)/3} = \kappa^2\]
and so
\[1 = \kappa^3 = \frac{1-r}{1+r}\kappa \implies \kappa = \frac{1+r}{1-r}.\]
It follows that $r^2 =-3$, which occurs if and only if $\alpha^{q+1} =1$, in which case $(\alpha +3\delta_\kappa^{-q})^{q+1} =1 \iff \kappa^{3(q+1)} =1$. Hence when $r = \sqrt{-3}$, the $q$ values of $\kappa$ such that $\kappa^{q+1} =1$ and $\kappa^3 \ne 1$ give $q$ distinct solutions $\delta = \delta_\kappa$ to $(a +3\delta^{-q})^{q+1}$.

Now suppose $q \equiv 2 \mod 3$. Then $(\alpha +3\delta_\kappa^{-q})^{q+1} =1 \iff \kappa^3(\kappa^{3(q-1)}(r -1)^3 +(r+1)^3) =0$. Since $r^q=-r$, we have $\left(\frac{1+r}{1-r}\right)^{q+1}=1$, and so there exist $q-1$ elements $\kappa \in \F_{q^2}$ such that $\kappa^{q-1} = \frac{1+r}{1-r}$. Note again that $\delta_\kappa = \delta_\iota$ if and only if $\kappa^3 = \iota^3$. Since $\kappa^{q-1}=\iota^{q-1}$ and $q \equiv 2 \mod 3$, the $q-1$ values of $\kappa$ such that $\kappa^{q-1} = \frac{1+r}{1-r}$ give $q-1$ distinct solutions $\delta = \delta_\kappa$ to $(a +3\delta^{-q})^{q+1}$, provided $\kappa^3 \ne 1$. If $\kappa^3 =1$, then
\[\frac{1+r}{1-r} = \kappa^{q-1} = \kappa(\kappa^3)^{(q-2)/3} = \kappa.\] It follows that $r^2 =-3$, which occurs if and only if $\alpha^{q+1} =1$, in which case $(\alpha +3\delta_\kappa^{-q})^{q+1} =1 \iff \kappa^3(\kappa^{3(q-1)}-1) =0$. Hence when $r = \sqrt{-3}$, the $q-2$ values of $\kappa$ such that $\kappa^{q-1} =1$ and $\kappa^3 \ne 1$ give $q-2$ distinct solutions $\delta = \delta_\kappa$ to $(a +3\delta^{-q})^{q+1}$.

The remaining two solutions to $(\alpha +3\delta^{-q})^{q+1} = 1$ for both the case in which $\alpha^{q+1} \ne 1$ and the case in which $\alpha^{q+1} = 1$ are given by $\delta = \delta_0$ and $\delta = \frac{-3\alpha}{2}(1+r)$.

Thus the claim holds and hence $P_{\delta,\alpha}(x)$ is reducible.
\end{proof}

Next we determine precisely when different values of $(\delta,\alpha)$ define the same polynomial $P_{\delta,\alpha}(x)$.

\begin{lemma}\label{lem:pdadistinct}
Suppose $P_{\delta,\alpha}(x)=P_{\delta',A}(x)$ for $(\delta,\alpha)\ne (\delta',A)$. Then $P_{\delta,\alpha}(x)=(x-\delta/3)^3$.
\end{lemma}

\begin{proof}
By comparing coefficients of $P_{\delta,\alpha}(x)$ and $P_{\delta,A}(x)$, we see that $\delta = \delta'$, so $\alpha \ne A$. Then
\[\delta = \frac{3(A^{1-q}-\alpha^{1-q})}{\alpha-A}\]
and
\[K := \alpha^{2(1-q)} -\alpha^{2-q}A +\alpha^{1-q}A^2 +A^{2(1-q)} +(\alpha^2 -2\alpha^{1-q})A^{1-q} -\alpha A^{2-q} =0.\]
We calculate that \[P_{\delta,\alpha}(x) = x^3 -\frac{3(A^{1-q}-\alpha^{1-q})}{\alpha-A}x^2 -\frac{3(\alpha A^{1-q}-A\alpha^{1-q})}{\alpha-A}x -\frac{(\alpha^2 -\alpha^{1-q})A^{1-q} -(\alpha A -\alpha^{1-q})\alpha^{1-q}}{\alpha-A}\]
and
\[\left(x-\frac{\delta}{3}\right)^3 = x^3 -\frac{3(A^{1-q}-\alpha^{1-q})}{\alpha-A}x^2 +\frac{3(A^{1-q} -\alpha^{1-q})^2}{(\alpha-A)^2}x -\frac{(A^{1-q}-\alpha^{1-q} )^3}{(\alpha -A)^3}.\]
The difference of these two polynomials is
\[
-\frac{3K}{(\alpha-A)^2}x -\frac{(\alpha^2 +\alpha^{1-q} -\alpha A -A^{1-q})K}{(\alpha-A)^3}=0,
\]
and so the result holds.
\end{proof}

We are now ready to enumerate the number of irreducible polynomials of the form $P_{\delta,\alpha}(x)$ which satisfy Condition (1).

\begin{lemma}\label{lem:pdacount}
The number of polynomials of the form $P_{\delta,\alpha}(x)$ which are irreducible and satisfy Condition (1) is $\frac{(q+1)(q-3)(q^2-1)}{3}$ when $q \equiv 1 \mod 3$, and $\frac{(q+1)(q-1)(q^2-1)}{3}$ when $q \equiv 2 \mod 3$.

Moreover, the number of polynomials of the form $P_{\delta,1}(x)$ which are irreducible and satisfy Condition (1) is $\frac{2(q^2-1)}{3}$.
\end{lemma}

\begin{proof}
For each $\alpha$, we wish to determine the number of $\delta$ such that $P_{\delta,\alpha}(x)$ is irreducible. If $\alpha^{q+1} =4$, then $P_{\delta,\alpha}(x) = (x-(\delta +\alpha))(x+\alpha/2)^2$ is reducible. We fix $\alpha$ such that $\alpha^{q+1} \ne 4$ and count the number of $\delta$ for which $P_{\delta,\alpha}(x)$ is reducible.

Suppose $P_{\delta,\alpha}(x)$ is reducible. Then $P_{\delta,\alpha}(x) = (x-\tau)(x-\sigma)(x-\nu)$ for some $\tau,\sigma,\nu \in \Fqt$ by Lemma \ref{lem:pdadisc}. Equating coefficients yields that 
\begin{align*}
    \tau+\sigma+\nu &=\delta, \tag{i}\\
    -(\tau\sigma+\tau\nu+\sigma\nu)&=\delta\alpha+3\alpha^{1-q},\tag{ii}\\
    \tau\sigma\nu&=\delta\alpha^2(1-\alpha^{-(q+1)})/3+\alpha^{2-q}.\tag{iii}
\end{align*}
We obtain that (up to labelling of $\sigma$ and $\nu$)
\[\sigma = -\left(\frac{a\tau+\alpha^{1-q}}{\tau+b}\right)\]
and
\[\nu = -\left(\frac{b\tau+\alpha^{1-q}}{\tau+a}\right),\]
where $a$ and $b$ are as in (\ref{eqn:abalpha}), and $\tau\notin\{-a,-b\}$. Note that if $\tau\in\{-a,-b\}$, then $\alpha^{q+1}=4$, contrary to our assumption. Note furthermore that $a\ne b$ precisely when $\alpha^{q+1}\ne 4$.

We remark that $\tau = \sigma$ if and only if $\tau^2 +\alpha\tau +\alpha^{1-q}=0$, while $\tau = \nu$ if and only if $\tau^2 +\alpha\tau +\alpha^{1-q}=0$, and $\sigma = \nu$ if and only if $\tau^2 +\alpha\tau +\alpha^{1-q}=0$ or $a=b$. Hence if any two of $\tau,\sigma$ and $\nu$ are equal, then all three are equal and $P_{\delta,\alpha}(x)= (x-\tau)^3$ for some $\tau\in \Fqt$. Equations (i) and (ii) then imply that $\tau^2+\alpha\tau+\alpha^{1-q}=0$, and (iii) is satisfied whenever (i) and (ii) are satisfied, since it can be rearranged to read $(\tau^2+\alpha\tau+\alpha^{1-q})(\tau-\alpha)=0$. The discriminant of $\tau^2+\alpha\tau+\alpha^{1-q}$ is $\alpha^2(1-4\alpha^{-(q+1)})$, which is nonzero by assumption and always a square in $\Fqt$, so there are precisely two values of $\tau$, and hence two values of $\delta$, for which $P_{\delta,\alpha}(x)$ has a triple root in $\Fqt$.

Hence for any of the $q^2-4$ values of $\tau$ such that $(\tau+a)(\tau+b)(\tau^2+\alpha\tau+\alpha^{1-q})\ne 0$, there is a unique $\delta$ for which $\tau$ is a root of a polynomial $P_{\delta,\alpha}(x)$ having three distinct roots in $\Fqt$. Therefore there are $\frac{q^2-4}{3}$ values of $\delta$ for which $P_{\delta,\alpha}(x)$ has three distinct roots in $\Fqt$.

Hence there are $q^2-2-\frac{q^2-4}{3}=\frac{2(q^2-1)}{3}$ values of $\delta$ for which $P_{\delta,\alpha}(x)$ is irreducible. Recall from Theorem \ref{thm:character} that $P_{\delta,\alpha}(x)$ satisfies Condition (1) if and only if $\frac{4-\alpha^{q+1}}{3\alpha^{q+1}}$ is a nonzero square in $\Fq$, and $\delta=0$ or $(\alpha+3\delta^{-q})^{q+1}\ne 1$. By Lemma \ref{lem:pdared}, it cannot occur that $P_{\delta,\alpha}(x)$ is irreducible when $\frac{4-\alpha^{q+1}}{3\alpha^{q+1}}$ is a nonzero square in $\Fq$ and $(\alpha+3\delta^{-q})^{q+1}= 1$, and hence it remains only to count the number of values of $\alpha$ for which $\frac{4-\alpha^{q+1}}{3\alpha^{q+1}}$ is a nonzero square in $\Fq$. Each such $\alpha$ will contribute $\frac{2(q^2-1)}{3}$ irreducibles satisfying Condition (1); in particular for $\alpha=1$ we get the second claim.

Suppose $\frac{4-\alpha^{q+1}}{3\alpha^{q+1}}=y^2$ for some $y \in \Fq^*$. If $y^2 \ne -1/3$, then
\[\alpha^{q+1} = \frac{4}{3y^2 +1}.
\]
Since $-3$ is a square in $\Fq$ if and only if $q \equiv 1 \mod 3$, we have
\[
\left|\{y^2 : y \in \Fq \mid y^2 \ne -1/3\}\right| = \begin{cases}
(q-3)/2 & \text{if } q \equiv 1 \mod 3\\
(q-1)/2 & \text{if } q \equiv 2 \mod 3
\end{cases}.
\]
The number of such $\alpha$ is hence $(q+1)(q-3)/2$ when $q \equiv 1 \mod 3$, and $(q+1)(q-1)/2$ when $q \equiv 2 \mod 3$, completing the proof.
\end{proof}

Next we enumerate the number of irreducible polynomials of the form $Q_{\delta,\gamma}(x)$ which satisfy Condition (1).

\begin{lemma}\label{lem:Qcount}
The number of polynomials of the form $Q_{\delta,\gamma}(x) = x^3 -\delta x^2 -\gamma x +\delta\gamma/9$ that are irreducible and satisfy Condition (1) is  $\frac{(q-1)(q+1)^2}{3}$.
\end{lemma}

\begin{proof}
First note that there are $q^2(q+1)/2$ polynomials of the form $Q_{\delta,\gamma}(x)$ satisfying Condition (1); there are $q^2$ choices for $\delta$ and $(q+1)/2$ choices for $\gamma$, since $\gamma^{(q+1)/2} =3$. We can transform $Q_{\delta,\gamma}(x)$ into a cubic
\[Q'(y) = y^3 -(\delta^2/3 +\gamma) y -2\delta(\delta^2 +3\gamma)/27\]
whose coefficient of $y^2$ is zero via the change of variable $y = x-\delta/3$. Then, using the notation in Lemma \ref{lem:dickson}, we require
\[R = \frac{4\gamma}{9}\left(\delta^2 +3\gamma\right)^2\]
to be a nonzero square in $\Fqt$ in order for $Q'(y)$ to be irreducible. Since
\[\gamma^{(q^2 -1)/2} = (\gamma^{(q+1)/2})^{q-1} = 3^{q-1} = 1,\]
we have that $\gamma$, and hence $R$, is a square in $\Fqt$. To ensure $R$ is nonzero, we need $\delta^2 \neq -3\gamma$. We now have
\[\mu = \pm \frac{\sqrt{R}}{9} = \pm 2\sqrt{\gamma}\left(\frac{\delta^2 +3\gamma}{27}\right)\]
and so for irreducibility of $Q'(y)$ we require
\begin{align*}
S &= \frac{1}{27}(\delta \pm \sqrt{-3\gamma})(\delta + \sqrt{-3\gamma})(\delta - \sqrt{-3\gamma})\\
&=\frac{\delta \mp \sqrt{-3\gamma}}{\delta \pm \sqrt{-3\gamma}}\left(\frac{\delta \pm \sqrt{-3\gamma}}{3}\right)^3
\end{align*}
to be a noncube in $\Fqt$. Thus, we need ($\delta \mp \sqrt{-3\gamma})/(\delta \pm \sqrt{-3\gamma})$ to be a noncube. Since 
\[C := \frac{\delta - \sqrt{-3\gamma}}{\delta + \sqrt{-3\gamma}}\] is a cube if and only if
\[\frac{1}{C} = \frac{\delta + \sqrt{-3\gamma}}{\delta - \sqrt{-3\gamma}}\] is a cube, we proceed with determining when $C$ is a cube without loss of generality. Let $x \in \Fqt$. Then
\begin{align*}
&C = \frac{\delta - \sqrt{-3\gamma}}{\delta + \sqrt{-3\gamma}} \cdot \frac{\delta - \sqrt{-3\gamma}}{\delta - \sqrt{-3\gamma}} = x^3\\
\iff &(x^3 -1)\delta^2 +2\sqrt{-3\gamma}\delta +(x^3 -1)3\gamma =0\\
\iff &\delta = \sqrt{-3\gamma} \text{ or } \delta = -\left(\frac{x^3 +1}{x^3 -1}\right)\sqrt{-3\gamma}.
\end{align*}
If $\delta = \sqrt{-3\gamma}$ then $\delta^2 = -3\gamma$. Note that
\[-\sqrt{-3\gamma} = -\sqrt{-3\sigma} \iff \gamma = \sigma\]
for $\gamma, \sigma \in \Fqt$ with $\gamma^{(q+1)/2} = \sigma^{(q+1)/2} =3$ and that
\[\frac{x^3 +1}{x^3 -1}\phi = \frac{y^3 +1}{y^3 -1}\phi \iff x^3 = y^3\]
for $x,y,\phi \in \Fqt$ with $\phi \neq 0$. There are $(q^2 -1)/3$ nonzero cubes in $\Fqt$. When $x=0$, $\delta = \sqrt{-3\gamma}$. Hence the number of pairs $(\delta,\gamma)$ that yield a reducible $Q_{\delta,\gamma}(x)$ is
\begin{align*}
&\left|\left\{\left(\left(\frac{x^3 +1}{x^3 -1}\right)\sqrt{-3\gamma},\gamma\right) : x,\gamma \in \Fqt \,\middle|\, x \ne 0, \gamma^{(q+1)/2} = 3 \right\}\right| + \left|\left\{\left(\sqrt{-3\gamma},\gamma\right) : \gamma \in \Fqt \,\middle|\, \gamma^{(q+1)/2} = 3 \right\}\right|\\
=& \left(\frac{q^2 -1}{3}\right)\left(\frac{q+1}{2}\right) + \frac{q+1}{2}\\
=& \frac{(q+1)(q^2 +2)}{6}.
\end{align*}
Since different pairs $(\delta,\gamma)$ clearly define different polynomials $Q_{\delta,\gamma}(x)$ polynomials, the number of irreducibles of the form $Q_{\delta,\gamma}(x)$ is
\[\frac{q^2(q+1)}{2}-\frac{(q+1)(q^2 +2)}{6} = \frac{(q-1)(q+1)^2}{3}.\]
\end{proof}

Finally we enumerate the number of irreducible polynomials of the form $B_{\theta}(x)$ which satisfy Condition (1).

\begin{lemma}\label{lem:Bcount}
The number of polynomials of the form $B_\theta(x) = x^3 -\theta$ that are irreducible and satisfy Condition (1) is $\frac{2(q^2-1)}{3}$ when $q \equiv 1 \mod 3$, and zero otherwise. 
\end{lemma}

\begin{proof}
By Theorem \ref{thm:binomialirredcond1}, it suffices to count the number of elements $\theta \in \Fqt$ such that 3 does not divide $\frac{q^2-1}{o(\theta)}$. Let $\Fqt^* = \left<\sigma\right>$ and suppose that $3 \mid \frac{q^2 -1}{o(\theta)}$. Then $\frac{q^2 -1}{o(\theta)} = 3k$ for some $k \in \ZZ$, so $o(\theta) = \frac{q^2 -1}{3k}$ and thus $\theta \in \left<\sigma^3\right>$. Hence there are $\left|\left<\sigma\right>\right|-\left|\left<\sigma^3\right>\right| = \frac{2(q^2-1)}{3}$ elements $\theta$ such that $3 \nmid \frac{q^2-1}{o(\theta)}$.
\end{proof}

Combining Lemmas \ref{lem:pdacount}, \ref{lem:Qcount}, and \ref{lem:Bcount} gives us the following. This enumeration will allow us in the next section to fully count and characterise the equivalence classes.

\begin{corollary}\label{cor:countall}
The total number of irreducible cubic polynomials in $\Fqt[x]$ satisfying Condition (1) is
\[\begin{cases}
\frac{q(q-1)^2(q+1)}{3} & \text{if } q \equiv 1 \mod 3\\
\frac{q(q-1)(q+1)^2}{3} & \text{if } q \equiv 2 \mod 3
\end{cases}.\]
\end{corollary}

\subsection{Equivalence representatives}

In order to calculate equivalence classes, we need to utilise the theory of {\it orbit polynomials}. Let $\Psi = \npmatrix{-b&-d\\c&a}\in \GL(2,q^2)$, and denote by $[\Psi]$ the corresponding element of $\PGL(2,q^2)$. Define a polynomial $F_\Psi(x)$ as follows:
\[
F_\Psi(x) = cx^{q^2+1}+ax^{q^2}+bx+d.
\]

Polynomials of this form have been studied extensively, for example in \cite{blakegaomullin}, \cite{ore}, \cite{stichtenothtopuzoglu}.

Given $s=[\Psi]\in \PGL(2,q^2)$ as above, define $s(x) = -\left(\frac{bx+d}{cx+a}\right)$. The {\it orbit polynomial} of the group $G$ generated by $s$ is defined as
\[
O_G(x) = \prod_{s\in G} (x-s(y))\in \FF_{q^2}(y)[x].
\]

The factorisation of polynomials of the form $F_\Psi(x)$ was determined in \cite{stichtenothtopuzoglu} and \cite{gowmcguire}.

\begin{theorem}
Let $s=[\Psi]=\left[\npmatrix{-b&-d\\c&a}\right] \in \PGL(2,q^2)$, and suppose $s$ has order $r$ dividing $q^2+1$. The irreducible factors of the polynomial $F_\Psi(x)$ of degree greater than two all have degree $r$, each of which are specialisations of $O_G(x)$ at some $y$.
\end{theorem}

We consider the case $\Psi = \npmatrix{-1&-1\\1&0}$, whence $F_1(x) := F_\Psi(x) = x^{q^2+1}+x+1$. The order of $s=[\Psi]$ is three, and 
\begin{align*}
O_G(x) &=  (x-y)(x-s(y))(x-s^2(y))\\
&= (x-y)\left(x+\frac{y+1}{y}\right)\left(x+\frac{1}{y+1}\right)\\
&= x^3+\left(\frac{1+3y-y^3}{y(y+1)}\right)x^2+\left(\frac{1-3y^2-y^3}{y(y+1)}\right)x-1\\
&= P_{\delta,1}(x),
\end{align*}
where $\delta = \frac{1+3y-y^3}{y(y+1)}$. Thus all irreducible cubic factors of $x^{q^2+1}+x+1$ over $\FF_{q^2}$ are of the form $P_{\delta,1}(x)$ for some $\delta$, and since there are precisely two roots of $x^{q^2+1}+x+1$ in $\FF_{q^2}$, we get $\frac{q^2-1}{3}$ such irreducible factors. Similarly, we can calculate that all irreducible cubic factors of $F_2(x) := x^{q^2+1}+x^{q^2}+1$ over $\FF_{q^2}$ are of the form $P_{\delta,1}(x)$ for some $\delta$. Since these polynomials cannot have any irreducible cubic factors in common, together with the count of the number of irreducibles of the form $P_{\delta,1}(x)$ performed in Lemma \ref{lem:pdacount}, we get the following.

\begin{theorem}
Every irreducible cubic polynomial of the form $P_{\delta,1}(x)$ is a factor of $F_1(x)F_2(x) = (x^{q^2+1}+x+1)(x^{q^2+1}+x^{q^2}+1)$, and every irreducible cubic factor of $F_1(x)F_2(x)$ is of the form $P_{\delta,1}(x)$.
\end{theorem}

Note that if $[\Psi]\ne[\Phi]$, then $F_\Psi(x)$ and $F_{\Phi}(x)$ can have at most a quadratic factor in common. Therefore if $P(x)$ divides $F_\Psi(x)$ and $Q(x)$ divides $F_{\Phi}(x)$ where $P$ and $Q$ have degree greater than two, then $P(x)$ and $Q(x)$ are equivalent if and only if $F_\Psi(x)$ and $F_{\Phi}(x)$ are equivalent. Moreover, any group element mapping $P(x)$ to $Q(x)$ must also map $F_\Psi(x)$ to $F_{\Phi}(x)$.

The element $\phi_{0,1}$ maps $F_2(x)$ to $F_1(x)$, and so every irreducible factor of $F_2(x)$ is equivalent to an irreducible factor of $F_1(x)$. Hence to calculate the equivalence classes amongst the polynomials of the form $P_{\delta,1}(x)$, it suffices to calculate equivalences between the divisors of $F_1(x)$ via elements of the stabiliser of $F_1(x)$ in $U$.

To this end, we now demonstrate how the action of the group $U$ manifests on polynomials of the form $F_{\Psi}(x)$.
\begin{lemma}
Let $\phi = \npmatrix{u^q&v\\v^q&u}$ with $u^{q+1}-v^{q
+1}\ne 0$. Then $F_\Psi^\phi(x) = (u^{q+1}-v^{q+1})F_{\phi^{-1}\Psi\phi}$.
\end{lemma}

\begin{proof}
We directly compute $F_\Psi^\phi$ as follows.
\begin{align*}
F_\Psi^\phi(x)&=(u+v^q x)^{q^2+1}F_\Psi\left(\frac{v+u^q x}{u+v^q x}\right) \\ &=(cu^{2q} +(a+b)u^qv^q +dv^{2q})x^{q^2 +1} +(au^{q+1} +cu^qv +duv^q +bv^{q+1})x^{q^2}\\
&\,\,\,+(bu^{q+1} +cu^qv +duv^q +av^{q+1})x +du^2 +(a+b)uv +cv^2\\
&=(u^{q+1}-v^{q+1})F_{\phi^{-1}\Psi\phi}(x),
\end{align*}
where the final equality holds since
\begin{align*}
\phi^{-1}\Psi\phi &= \frac{1}{u^{q+1}-v^{q+1}}\begin{pmatrix} u & -v\\ -v^q & u^q\end{pmatrix} \begin{pmatrix} -b & -d\\ c & a\end{pmatrix} \begin{pmatrix} u^q & v\\ v^q & u \end{pmatrix}\\
&=\frac{1}{u^{q+1}-v^{q+1}} \begin{pmatrix} -(bu^{q+1} +cu^qv +duv^q +av^{q+1}) & -(du^2 +(a+b)uv +cv^2)\\
cu^{2q} +(a+b)u^qv^q +dv^{2q} & au^{q+1} +cu^qv +duv^q +bv^{q+1}
\end{pmatrix}.
\end{align*}
\end{proof}

Next, we apply this to calculate the subgroup of $U$ stabilising $F_1(x)$, and hence permuting its irreducible cubic factors.

\begin{lemma}\label{lem:stabiliser}
The stabiliser of $F_1(x) = x^{q^2 +1} +x +1$ in $U$ is 
\[\{\phi_{u,
u^q-u} : u \in \Fqt^\times, u^{q-1} \ne (1 \pm \sqrt{-3})/2 \}.
\]
\end{lemma}

\begin{proof}
Let $\phi =\phi_{u,v}= \begin{pmatrix}
u^q & v\\
v^q & u
\end{pmatrix}$ with $u^{q+1} \neq v^{q+1}$ and let $\lambda \in \Fqt$. Then the matrix equation
\[\phi^{-1}\psi\phi = \lambda\psi\]
holds if and only if
\[\frac{1}{u^{q+1}-v^{q+1}}\begin{pmatrix}
-(u^{q+1} + uv^q +u^qv) & -(u^2 +uv +v^2)\\
(u^2 +uv +v^2)^q & v^{q+1} +uv^q +u^qv
\end{pmatrix} = \lambda\begin{pmatrix}
-1 & -1\\
1 & 0
\end{pmatrix}.\]
This equality holds if and only if $u^{q+1}\ne v^{q+1}$ and
\[v^{q+1} +uv^q +u^qv = 0 \tag{I}\]
and
\[(u^2 +uv +v^2)^q = u^{q+1} + uv^q +u^qv = u^2 +uv +v^2. \tag{II}\]
We now show that these  conditions are equivalent to $v^q = u-u^q$. First suppose $v^q = u-u^q$. Then equations (I) and (II) hold. Furthermore $u^{q+1}=v^{q+1}$ if and only if $(u-u^q)(u^q -u)= u^{q+1}$. Rearranging, we get $u^2(u^{2(q-1)} -u^{q-1} +1) = 0$, which occurs if and only if $u=0$ or $u^{q-1} = (1 \pm \sqrt{-3})/2$.

Now suppose (I) and (II) hold. If $v =0$, then (II) gives $u^{q+1} = u^2$, so $u=u^q$ and hence $v^q = -v = 0 = u-u^q$. If $v \ne 0$, we have $u^q = -(u+v)v^{q-1}$ from (I). Hence
\[u^{q+1} + uv^q +u^qv = u^2 +uv +v^2 \iff (u^2+uv+v^2)(v^q +v) =0\]
and
\[(u^2 +uv +v^2)^q = u^{q+1} + uv^q +u^qv = u^2 +uv +v^2 \iff (u^2+uv+v^2)(v^q +v)(v^q -v) =0.\]
If $0 = u^2 +uv +v^2 = u^{q+1} +uv^q +u^qv$, then $u^{q+1} = v^{q+1}$, which is not allowed. Thus $v^q = -v$. It follows from (I) that $v = u^q-u$. As before, the condition $u^{q+1}\ne v^{q+1}$ gives that $u\ne 0$ and $u^{q-1} \ne (1 \pm \sqrt{-3})/2$, completing the proof.
\end{proof}

This allows us to compute the number of projective equivalence classes amongst the polynomials $P_{\delta,1}(x)$, as well as the size of the union of these equivalence classes. As we will observe, this matches the total number of irreducible cubics satisfying Condition (1), implying that every equivalence class contains a polynomial of the form $P_{\delta,1}(x)$.

\begin{theorem}
\label{thm:totalcount}
The number of projective equivalence classes of irreducible polynomials of the form $P_{\delta,1}(x)$ is
\[
\left\{
\begin{array}{cc}
\frac{q-1}{3}&\textrm{if }q\equiv 1\mod 3,\\
\frac{q+1}{3}&\textrm{if }q\equiv 2\mod 3.
\end{array}\right.
\]
Moreover the number of monic irreducible polynomials projectively equivalent to some $P_{\delta,1}(x)$ is 
\[
\left\{
\begin{array}{cc}
\frac{q(q-1)(q^2-1)}{3}&\textrm{if }q\equiv 1\mod 3,\\
\frac{q(q+1)(q^2-1)}{3}&\textrm{if }q\equiv 2\mod 3.
\end{array}\right.
\]
\end{theorem}

\begin{proof}
Recall that in order to calculate the number of equivalence classes of polynomials of the form $P_{\delta,1}(x)$ satisfying Condition (1), it suffices to calculate the equivalence classes amongst the divisors of $F_1(x) = x^{q^2+1}+x+1$ under the stabiliser of $F_1(x)$. As shown in Lemma \ref{lem:stabiliser}, this consists of matrices of the form $\phi_{u,u^q-u}$ where  $u^2(u^{q-1}-u^{2(q-1)}-1)\ne 0$. 

There are $q^2-1$ such matrices when $q\equiv 1\mod 3$, and $(q-1)^2$ such matrices when $q\equiv 2\mod 3$, $q-1$ of which are scalar multiples of the identity. Therefore the divisors of $F_1(x)$ are partitioned into equivalence classes of size $q+1$ (resp. $q-1$) under this action when $q\equiv 1\mod 3$ (resp. $q\equiv 2\mod 3$), and so there are $\frac{q-1}{3}$ equivalence classes when $q\equiv 1\mod 3$ and $\frac{q+1}{3}$ equivalence classes when $q\equiv 2\mod 3$.

A further application of the Orbit-Stabiliser Theorem returns the claimed number of polynomials equivalent to some $P_{\delta,1}$.
\end{proof}

Choosing canonical representatives for each equivalence class among the $P_{\delta,1}$ polynomials is not straightforward. The following lemma establishes criteria for equivalence amongst polynomials of this shape.

\begin{lemma}\label{lem:epsilonvalues}
The polynomials $P_{\delta,1}(x)$ and $P_{\epsilon,1}(x)$ are projectively equivalent if and only if
\begin{align*}
\epsilon \in &\left\{\frac{9w(w-1) +\delta (w^3 -3w +1)}{w^3 -3w^2 +1 -\delta w(w-1)} : w^{q+1} =1, w \ne (1 \pm \sqrt{-3})/2 \right\}\\
&\,\,\,\cup \left\{\frac{-3(w^3 -3w^2 +1) -\delta(w^3 -3w +1)}{w^3 -3w +1 +\delta w(w-1)} : w^{q+1} =1, w \ne (1 \pm \sqrt{-3})/2 \right\}. 
\end{align*}
\end{lemma}

\begin{proof}
We have determined in this section that two polynomials of the form $P_{\delta,1}(x)$ are equivalent via $\phi_{u,u^q-u}$ or $\phi_{0,1}\phi_{u,u^q-u} = \phi_{u^q -u,u}$, where $u^{q-1} \ne (1 \pm \sqrt{-3})/2$.

First let $v = u^q -u$. Then by Corollary \ref{cor:polequiv}, $P_{\delta,1}(x)$ and $P_{\epsilon,1}(x)$ are equivalent if and only if
\[\lambda P_{\epsilon,1}(x) = (u(x+1) -u^qx)^3 P_{\delta,1}\left(\frac{u^q(x+1) -u}{u(x+1) -u^qx}\right).\]
Comparing coefficients of these polynomials yields that
\[\epsilon = \frac{9u^{q-1}(u^{q-1}-1) +\delta (u^{3(q-1)} -3u^{q-1} +1)}{u^{3(q-1)} -3u^{2(q-1)} +1 -\delta u^{q-1}(u^{q-1}-1)}.\]

Now let $u = v^q-v$. Then $P_{\delta,1}(x)$ and $P_{\epsilon,1}(x)$ are equivalent if and only if
\[\lambda P_{\epsilon,1}(x) = (v^q(x+1) -v)^3 P_{\delta,1}\left(\frac{v(x+1) -v^qx}{v^q(x+1) -v}\right).\]
Comparing coefficients again returns
\[\epsilon = \frac{-3(v^{3(q-1)} -3v^{2(q-1)} +1) -\delta(v^{3(q-1)} -3v^{q-1} +1)}{v^{3(q-1)} -3v^{q-1} +1 +\delta v^{q-1}(v^{q-1}-1)}.\]

Replacing $u^{q-1}$ and $v^{q-1}$ with $w$ in both expressions for $\epsilon$ gives the stated result.
\end{proof}

We now consider the question of when $P_{\delta,1}(x)$ is equivalent to $P_{\delta,1}^\sigma(x) = P_{\delta^q,1}(x)$. This is necessary in order to determine the equivalence classes (rather than projective equivalence classes). Furthermore this demonstrates that all of the $2$-spreads obtained have full automorphism group strictly larger than the group $C$.

\begin{lemma}\label{lem:projequiv}
Suppose $P_{\delta,1}(x)$ and $P_{\delta^q,1}(x)$ are irreducible and satisfy Condition (1). Then $P_{\delta,1}(x)$ and $P_{\delta^q,1}(x)$ are projectively equivalent. Hence two irreducible cubics satisfying Condition (1) are equivalent if and only if they are projectively equivalent.
\end{lemma}

\begin{proof}
By Lemma \ref{lem:epsilonvalues}, it suffices to show the existence of some $w \in \Fqt$ such that $w^{q+1}=1$ and
\[\delta^q = \frac{-3(w^3 -3w^2 +1) -\delta(w^3 -3w +1)}{w^3 -3w +1 +\delta w(w-1)}\]
or
\[\delta^q = \frac{9w(w-1) +\delta (w^3 -3w +1)}{w^3 -3w^2 +1 -\delta w(w-1)}.\]
Suppose the latter equality holds. Then
\[(\delta-\delta^q)w^3 +(\delta^{q+1} +3\delta^q +9)w^2 -(\delta^{q+1} +3\delta +9)w +\delta -\delta^q =0.\]
If $\delta = \delta^q$, then clearly $P_{\delta,1}(x)=P_{\delta^q,1}(x)$, and so we assume that $\delta \ne \delta^q$. Then we have
\[w^3 +\frac{\delta^{q+1} +3\delta^q +9}{\delta-\delta^q}w^2 -\frac{\delta^{q+1} +3\delta +9}{\delta -\delta^q}w +1 =0.\]
The left-hand side of this equation is a cubic polynomial in $\Fqt[w]$. Denote this polynomial by $f(w)$. Since $w^{3q}f(w^{-q}) = f(w)^q$, if $\tau$ is a root of $f(w)$ then so is $\tau^{-q}$. Hence if $f(w)$ is reducible, it must factorise as
\[(w-\tau)(w-\tau^{-q})(w-\nu),\]
where $\tau \in \Fqt$ and $\nu \in \F_{q^6}$. Since $-\nu\tau^{1-q} = 1$, it follows that $\nu = -\tau^{q-1} \in \Fqt$ and so $w=\nu$ is a solution to the equation with $w^{q+1}=1$. 

Hence it only remains to show that $f(w)$ cannot be irreducible. We apply a change of variables, and apply Lemma \ref{lem:dickson}. We obtain that $f(w)$ is irreducible if and only if $g(w)=w^3+sw+t$ is irreducible, where
\[s = -\frac{(\delta^2 +3\delta +9)^{q+1}}{3(\delta -\delta^q)^2}\]
and
\[t = -\frac{(\delta^2 +3\delta +9)^{q+1}(2\delta^{q+1} +3\delta^q +3\delta +18)}{27(\delta^q -\delta)^3} = \frac{2\delta^{q+1} +3\delta^q +3\delta +18}{9(\delta^q -\delta)}s.\]
Using the same notation as Lemma \ref{lem:dickson}, we calculate that
\[R = \frac{(\delta^2 +3\delta +9)^{2(q+1)}}{(\delta -\delta^q)^4}.\]
Setting $\mu = \pm\sqrt{R}/9$, then
\[S = \frac{(\delta^2 +3\delta +9)^{q+1}(2\delta^{q+1} +3(1 \pm \sqrt{-3})\delta^q +3(1 \mp \sqrt{-3})\delta +18)}{54(\delta^q -\delta)^3}.\]
If $q \equiv 2 \mod 3$ then
\[S = \left(\frac{((\delta^2 +3\delta +9)(\delta +3(1 \pm \sqrt{-3})/2))^{(q+1)/3}}{3(\delta^q -\delta)}\right)^3.\]
If $q \equiv 1 \mod 3$ then
\[S = \left(\frac{(\delta +3(1 \mp \sqrt{-3})/2)^{(2q+1)/3}(\delta +3(1 \pm \sqrt{-3})/2)^{(q+2)/3}}{3(\delta^q -\delta)}\right)^3.\]
Hence $S$ is always a perfect cube, and so $f(w)$ cannot be irreducible. Therefore $P_{\delta,1}(x)$ is always equivalent to $P_{\delta^q,1}(x)$.
\end{proof}

\begin{remark}
Note that this implies that the full stabiliser of the $2$-spread $\ell_\epsilon^C$ in $\GammaL(1,q^6)$ contains elements not in $C$, namely the map $x\mapsto x^{q^3}$.

However, this does not imply that every irreducible cubic satisfying Condition (1) is equivalent to a polynomial with coefficients in $\Fq$; in fact, counterexamples can be easily found already when $q=5$.

Finally, we remark that it is not true that all polynomials satisfying Condition (1) are equivalent if and only if they are projectively equivalent. We have counterexamples for polynomials of degree $5$; this will be the subject of future work.
\end{remark}

We summarise this section with our main result on equivalence classes.

\begin{corollary}\label{cor:allequiv}
Every irreducible cubic in $\Fqt[x]$ satisfying Condition (1) is equivalent to one of the form $P_{\delta,1}$. Furthermore, the number of equivalence classes of irreducible cubics satisfying Condition (1) is 
\[
\left\{
\begin{array}{cc}
\frac{q-1}{3}&\textrm{if }q\equiv 1\mod 3,\\
\frac{q+1}{3}&\textrm{if }q\equiv 2\mod 3.
\end{array}\right.
\]
\end{corollary}

\begin{proof}
This follows immediately from Corollary \ref{cor:countall}, Theorem \ref{thm:totalcount}, and Lemma \ref{lem:projequiv}.
\end{proof}

\section{Comparison with known results}

In this section we compare our results to the constructions and partial classifications which follow from the previous work of \cite{bartolitimpanella} and \cite{FL}.

\subsection{Results of Bartoli-Timpanella}

Recall from Lemma \ref{lem:bartoli} that $f_{a,b}(X) = X(1 +aX^{q(q-1)} +bX^{2(q-1)})$ is a permutation polynomial of $\Fqt$ if and only if $P(x) = x^3 +b^{-1}x +ab^{-1}$ satisfies Condition (1). In \cite{bartolitimpanella} the following was shown.

\begin{theorem}[\cite{bartolitimpanella}, Main Theorem]
Let $p>3$ be a prime and $q = p^h$, with $h \geq 1$. Then $f_{a,b}(X)$ is a permutation polynomial of $\Fqt$ if and only if either
\[\tag{PP1}\begin{cases}
a^qb^q = a(b^{q+1} -a^{q+1})\\
1 -4(ba^{-1})^{q+1} \text{ is a square in } \Fq^*,
\end{cases}\]
or
\[\tag{PP2}\begin{cases}
a^{q-1} +3b =0\\
-3(1 -4(ba^{-1})^{q+1}) \text{ is a square in } \Fq^*.
\end{cases}\]
\end{theorem}

We now compare the characterization of permutation polynomials of the form $f_{a,b}(X)$ with our characterization of polynomials satisfying Condition (1). Note that $P(x) = x^3 +b^{-1}x +ab^{-1}$ cannot be of the form $B_\theta(x)$ nor $Q_{\delta,\alpha}(x)$. Hence if $P(x)$ is irreducible and satisfies (PP1) or (PP2), then it must be of the form $P_{\delta,\alpha}(x)$. Thus we must have $\delta=0$, $a = \alpha/3$ and $b = -\alpha^{q-1}/3$.

With these parameters, Condition (PP1) becomes
\[\begin{cases}
-\alpha/9 = \alpha(1-\alpha^{q+1})/27\\
1 -4\alpha^{-(q+1)} \text{ is a square in } \Fq^*.
\end{cases}\]
The equality holds if and only if $\alpha^{q+1} =4$, in which case $P_{\delta,\alpha}(x)$ is reducible, contradicting our assumptions. Hence any polynomial satisfying (PP1) must be reducible.

Under the same criteria, Condition (PP2) is now
\[\tag{PP2}\begin{cases}
0=0\\
-3(1 -4\alpha^{-(q+1)}) \text{ is a square in } \Fq^*.
\end{cases}\]
Since $\delta = 0$ and $-3(1 -4\alpha^{-(q+1)})$ is a square in $\Fq^*$ if and only if $-(1-4\alpha^{-(q+1)})/3 = \frac{4-\alpha^{q+1}}{3\alpha^{q+1}}$ is a square in $\Fq^*$, Condition (PP2) agrees with the conditions in Theorem \ref{thm:character} for an irreducible polynomial of the form $P_{\delta,\alpha}(x)$ to satisfy Condition (1).

\subsection{Results of Feng-Lu}

Recall that in \cite{FL}, the polynomials
\[
g_{3,\rho}(x) = x^3 -3x+(\rho+\rho^q),
\]
were shown to be irreducible and satisfy Condition (1) when $\rho$ has order $q+1$. Such a polynomial lies in $\Fq[x]$. We now show that our classification contains examples not equivalent to any of those constructed in \cite{FL}.

\begin{lemma}
Every polynomial of the form $g_{3,\rho}(x)$ is equivalent to one of the form $P_{\delta,1}(x)$. 
    Not every irreducible polynomial of the form $P_{\delta,1}(x)$ is equivalent to one of the form $g_{3,\rho}(x)$.
\end{lemma}

\begin{proof}
It is immediate to verify that $g_{3,\rho}(x) = x^3 -3x+(\rho+\rho^q) = P_{0,-(\rho+\rho^q)}(x)$. From Corollary \ref{cor:allequiv}, this is equivalent to some $P_{\delta,1}(x)$, proving the first claim.

It is straightforward to see that $g_{3,\rho}(x)=g_{3,\rho^q}(x)$, and that $g_{3,\rho}(x)$ and $g_{3,-\rho}(x)$ are equivalent via $\phi_{u,0}$ with $u^{q-1}=-1$. Hence the number of equivalence classes of polynomials of the form $g_{3,\rho}(x)$ is at most $\frac{q+1}{4}$, and by Corollary \ref{cor:allequiv}, the second claim holds.
\end{proof}

\subsection{Conclusion}
In this paper we have fully characterised and classified cyclic $2$-spreads in $V(6,q)$ up to equivalence, and hence classified a class of flag-transitive linear spaces with assumed automorphism group. The classification includes new examples. 

\bibliographystyle{abbrv}
\bibliography{cycliclinespreads}
\end{document}